\documentclass{amsart}
\usepackage{amsmath, amsthm, amssymb}
\input xy
\xyoption{all}

\theoremstyle{plain} 
\newtheorem{theorem} {Theorem}
\newtheorem{proposition} {Proposition}
\newtheorem{lemma} {Lemma}
\newtheorem{corollary} {Corollary}

\theoremstyle{definition}

\newtheorem{definition} {Definition}

\theoremstyle{remark}
\newtheorem{remark} {Remark}

\newcommand{\mpair}[1]{\pair{\,#1\,}}

\newcommand{\mset}[1]{\set{\,#1\,}}

\newcommand{\pair}[1]{\langle #1\rangle}

\newcommand{\set}[1]{\{#1\}}

\newcommand{\COS}{{\text{\rm COS}}}
\newcommand{\inv}{^{-1}}

\newcommand{\nip}{np subset }
\newcommand{\nips}{np subsets }

\newcommand{\wh}{\widehat}
\newcommand{\wt}{\widetilde}
\newcommand{\nc}{\newcommand}

\nc{\on}{\operatorname}

\nc{\Z}{{\mathbb Z}}
\nc{\C}{{\mathbb C}}
\nc{\R}{{\mathbb R}}
\nc{\bbP}{{\mathbb P}}
\nc{\bF}{{\mathbb F}}

\nc{\boldD}{{\mathbb D}}
\nc{\oo}{{\mf O}}
\nc{\N}{{\mathbb N}}
\nc{\bib}{\bibitem}
\nc{\pa}{\partial}
\nc{\F}{{\mf F}}
\nc{\CA}{{\mathcal A}}
\nc{\CC}{{\mathcal C}}
\nc{\CE}{{\mathcal E}}
\nc{\CP}{{\mathcal P}}
\nc{\CO}{{\mathcal O}}
\nc{\CK}{{\mathcal K}}
\nc{\Ann}{\text{Ann}}
\nc{\Rad}{\text{Rad}}
\nc{\Res}{\text{Res}}
\nc{\Ind}{\text{Ind}}
\nc{\Ker}{\text{Ker}}
\nc{\id}{\text{id}}

\nc{\be}{\begin{equation}}
\nc{\ee}{\end{equation}}
\nc{\bee}{\begin{equation*}}
\nc{\eee}{\end{equation*}}
\nc{\RR}{{\mathcal R}}
\nc{\TT}{T}
\nc{\rh}{\vec}
\nc{\p}{{\wp}^{(\Gamma)}}
\makeatletter
\renewcommand{\p@enumi}{\thesubsection}
\makeatother
\newenvironment{conds}{
                       
                        \begin{enumerate} }
                     {\end{enumerate} }

\newenvironment{num}{
                      
                      \begin{enumerate} }
                    {\end{enumerate} }

\newcommand{\Int}{{\mathbb Z}}
\newcommand{\Nat}{{\mathbb N}}
\newcommand{\rat}{{\mathbb Q}}
\newcommand{\real}{{\mathbb R}}
\newcommand{\mc}{\mathcal}

\newcommand{\ck}[1]{{#1}^\vee}

\begin{document}

\title{Reflection subgroups of finite and affine Weyl groups}
\author{M.~J.  Dyer and G.~I. Lehrer}
\address{School of Mathematics and Statistics\\
University of Sydney. Sydney. 2006. Australia.}
\subjclass{20F55, 51F15}
\begin{abstract}
   We discuss the classification of
   reflection subgroups of finite and affine Weyl groups
   from the point of view of their root systems. A short case free proof is given of the 
   well known classification of the isomorphism classes of reflection
   subgroups using completed Dynkin diagrams, for which there seems to be no 
   convenient source in the literature. This is used as a basis for 
   treating the affine case, where finer classifications of reflection
   subgroups are given, and combinatorial aspects of root systems are shown to appear.
   Various parameter sets for certain types of subsets of roots are interpreted in terms of
   alcove geometry and the Tits cone, and combinatorial identities are derived.
 \end{abstract}
\maketitle

\section{Introduction}\label{intro}  
\subsection{}\label{ss1.1} Let $V$ be a Euclidean space, that is, 
a real vector space with a positive definite symmetric bilinear form
which we write $\mpair{-,-}\colon V\times V\rightarrow \real$. 
Let $\Phi$ be a finite reduced root system in $V$, in the sense of \cite{Bour}.
The following basic facts may be found in \cite{Bour} or \cite{St}.
For $\alpha\in\Phi$, we have the corresponding coroot $\ck\alpha=\frac{2\alpha}{\mpair{\alpha,\alpha}}$.
The set $\ck\Phi:=\{\ck\alpha\mid \alpha\in\Phi\}$ is again a reduced root
system in $V$, called the root system dual to $\Phi$. 

For $\alpha\in\Phi$ the corresponding reflection $s_\alpha:V\to V$ is defined 
by 
\be
s_\alpha(v)=v-\mpair{\ck\alpha,v}\alpha\;\;\;\;(v\in V).
\ee
The Weyl group $W=W(\Phi)$ of $\Phi$ is the group generated by the reflections
$s_\alpha$, $\alpha\in\Phi$, and it is well known that any reflection
in $W$ is of the form $s_\alpha$ for some $\alpha\in\Phi$.
In this work we shall assume that $W$ (or $\Phi$) is crystallographic;
this means that $\mpair{\alpha,\ck \beta}\in \Int$ for all $\alpha,\beta\in \Phi$.

A subset $\Phi_+\subseteq\Phi$ is a {\em positive system} if it consists of the
 roots of $\Phi$ which are positive relative to some vector space total ordering of $V$. Such an
ordering is defined by its cone $V_+\subseteq V$ of `positive elements',
which is closed under positive linear combinations, and satisfies $V=V_+\cup-V_+\cup\{0\}$
(disjoint union).
A subset $\Lambda\subseteq\Phi$ is a {\em simple system} if $\Lambda$ is linearly
independent and every root of $\Phi$ is a linear combination of $\Lambda$ in 
which all nonzero coefficients have the same sign, i.e. are either all positive or all
negative. It is well known that each 
positive system contains a unique simple system $\Lambda\subseteq\Phi_+$, and 
conversely. Moreover the positive systems (and hence the simple systems) 
of $\Phi$ are in bijection
with $W$, which acts on them regularly; that is, given two simple systems
$\Lambda_1,\Lambda_2$ of $\Phi$, there is a unique element $w\in W$ such
that $\Lambda_2=w\Lambda_1$.

\subsection{} A {\em root subsystem} of $\Phi$ is a subset $\Psi\subseteq\Phi$ which is itself a root
system. This is equivalent to the requirement that for $\alpha,\beta\in\Psi$,
$s_\alpha(\beta)\in\Psi$.

{\em Positive subsystems} and {\em simple subsystems} of $\Phi$ are respectively
defined as the positive systems and simple systems of a root subsystem of $\Phi$.

A {\em reflection subgroup} $W'$ of $W$ is a subgroup which is generated by a subset of
the reflections $s_\alpha$ ($\alpha\in\Phi$). Any such subgroup has an associated set
$\Psi=\Psi(W')$ of roots of $\Phi$, where $\Psi=\{\alpha\in\Phi\mid s_\alpha\in W'\}$.
Then $\Psi$ is clearly a root subsystem of $\Phi$, and the correspondence
$W'\leftrightarrow\Psi(W')$ is a bijection between reflection subgroups of $W$
and root subsystems of $\Phi$. In this work, we shall discuss the classification  
of root subsystems under various types of equivalence (isomorphism, $W$-conjugacy
or equality), for both the finite cases and the associated affine root systems. 
Our treatment therefore includes a 
classification of reflection subgroups under the corresponding equivalence.

\subsection{} A root subsystem $\Psi\subseteq\Phi$ is {\em closed} if, whenever $\alpha,\beta\in\Psi$
and $\alpha+\beta\in\Phi$ then $\alpha+\beta\in\Psi$.

The closed subsystems of a finite root system $\Phi$ were classified by Borel 
and de Siebenthal
in \cite{BS}, using the connection between crystallographic Weyl groups
and Lie algebras. The classification of all subsystems may be deduced from this.
This was also done in \cite{D}, and a treatment may be found in \cite{K};
a general classification for all unitary (complex) reflection groups
may be found in \cite{LT}. In \cite{Ca}, the classification of
reflection subgroups of finite Weyl groups $W$ is used in the
determination of the conjugacy classes of $W$, and the reference given for this result
is to \cite{BS}. It therefore seems reasonable to begin with a clear, concise
and self contained treatment of the finite case. We do this in \S\ref{s-finite}.

Our results in the case of the affine
Weyl groups have some slight overlap with those of \cite{F-T}, although
the methods are different. 
Background on affine Weyl groups is collected in \S \ref{s:aff}. 
We   give two bijective  parameterisations of the   reflection subgroups of affine Weyl groups. 
Their parameterisation  in \S \ref{s1} (by  specifying their ``canonical simple system'') 
generalizes the classification of reflection subgroups of the infinite dihedral group by 
specifying their set of  generating reflections of minimal total length. Their parameterisation 
in \S \ref{s5} (by explicitly describing their root system) generalizes the classification of 
non-trivial  reflection subgroups of the infinite dihedral group by specifying their reflections  by  a
coset  in $\Int$ of a subgroup of $\Int$. One or the other   parameterisation  may be 
more advantageous for a particular application (e.g computing the index of a reflection 
subgroup, or  a conjugate of it); the relation between the parameterisations is considered 
in detail in \S \ref{comparison}.

\section{The case of a finite root system}\label{s-finite}
 
We begin with the following well known
characterisation of the simple subsystems of $\Phi$.

 \begin{lemma}\label{lem:finsimple} 

 Let $\Phi$ be a finite root system in the Euclidean space $V$.
 A subset $\Gamma\subseteq\Phi$ is a simple subsystem of $\Phi$ 
 (i.e. a simple system in a root subsystem of $\Phi$) if and only
 if 
 \begin{conds}
 \item $\Gamma$ is linearly independent.
 \item For any pair $\alpha,\beta\in\Gamma$, $\mpair{\alpha,\ck\beta}\leq 0$.
 \end{conds}
 \end{lemma}
 
\begin{proof} If $\Gamma$ is a simple subsystem of $\Psi\subseteq\Phi$, then
by definition $\Gamma$ is linearly independent, and the property (ii)
is standard (cf. \cite[Appendix, Lemma (10)]{St}).

Conversely, let $\Gamma\subseteq\Phi$ satisfy the conditions (i) and (ii)
and write $W'$ for the group generated by the reflections $s_\gamma$ ($\gamma\in\Gamma$).
Let $\Psi$ be the set of roots $W'\Gamma\subseteq\Phi$. Then $\Psi$ is a 
crystallographic root system with Weyl group $W'$,
for if $\alpha\in\Psi$, $\alpha=w\gamma$
for some $w\in W'$ and $\gamma\in\Gamma$; hence 
$s_\alpha=s_{w\gamma}=ws_\gamma w\inv\in W'$, and so
$\Psi$ is precisely the set of roots $\alpha\in\Phi$ such
that $s_\alpha\in W'$.
Replacing $W'$ with $W$ and $\Psi$ with $\Phi$, we may therefore suppose
that the reflections $s_\gamma$ ($\gamma\in\Gamma$) generate $W$,
and that $\Phi=W\Gamma$.

Let $\Phi_+$ be the roots of $\Phi$ which are positive linear
combinations of $\Gamma$. 
If $\dot \alpha$ denotes the unit vector in the direction of $\alpha$,
the assumptions imply that if $\gamma_1,\gamma_2\in\Gamma$,
then $\mpair{\dot \gamma_1,\dot \gamma_2}\in\{-\cos\frac{\pi}{2},
-\cos\frac{\pi}{3},-\cos\frac{\pi}{4},-\cos\frac{\pi}{6}\}$. 
It now follows from \cite[Lemma (4.2)]{DyRef} that $\Phi=\Phi_+
\amalg-\Phi_+$.
\end{proof}

We shall also require the following result.

\begin{lemma}\label{lem:negchamber} 
 Let $\Phi$ be an indecomposable finite crystallographic
 root system in the Euclidean space $V$, and let
 $\Gamma\subseteq\Phi$ be a simple subsystem of $\Phi$. 
Let $\CC$ be the corresponding (closed) fundamental chamber
for the action of $W=W(\Phi)$ on $V$. Then $\alpha\in\CC\cap\Phi$ if and
only if $\alpha$ is the highest root $\omega$ of $\Phi$ or 
$\alpha=\ck{\omega^*}$,
where $\omega^*$ is the highest root of $\ck\Phi$
\end{lemma}
\begin{proof}
It follows from arguments based on the fact that 
for any two roots $\alpha_1,\alpha_2\in\Phi$, we have
$\mpair{\alpha_1,\alpha_2}\mpair{\alpha_2,\alpha_1}=
4\cos^2\phi,$ where $\phi$ is the angle between $\alpha_1$
and $\alpha_2$, that there are at most two possibilities
for the length of a root in $\Phi$ (see \cite[Ch. VI, \S 1.3]{Bour}).

It is also standard that $W$ acts transitively on the set of roots
of a given length. Now it is clear that the roots $\omega,\ck{\omega^*}$ 
of the statement lie in $\CC$, and by the regularity of the action of $W$
on the set of chambers, they are the only roots in their $W$ orbit
in $\CC$. But the highest root is always long, whence the $W$-orbit
of $\omega$ and $\ck{\omega^*}$ include all roots of $\Phi$. This proves the
Lemma.
\end{proof}

Note that in the above lemma, $\omega$ and $\ck{\omega^*}$
have the same length (and hence coincide) if and only
if $\Phi$ has just one root length. In that case, 
we declare all roots of $\Phi$ to be both short and long.

The root subsystem $\Psi\subseteq\Phi$ is said to be {\em parabolic}
if $\Phi$ has a simple system $\Pi$ such that there is a subset 
$\Gamma\subseteq\Pi$ which is a simple system for $\Psi$. The 
parabolic subgroups (i.e. the Weyl groups of parabolic subsystems)
are known to be precisely the stabilisers in $W$ of points
of $V$. In general $W$ will have reflection subgroups other than 
parabolic ones; the next definition provides a convenient way
of describing them.

\begin{definition}\label{def:elemext}
Let $\Psi\subsetneq\Phi$ be a root subsystem. We say that 
$\Phi$ is an {\em elementary extension} of $\Psi$ if 
there is a simple system $\Pi$ and 
an indecomposable component $\Sigma$ of $\Phi$
such that a simple system for $\Psi$ 
is obtained by deleting one element of $\Sigma$ from $\Pi\cup\{-\theta\}$,
where $\theta$ is a root of $\Sigma$ which lies in the chamber 
$\CC$ defined by the positive system corresponding to $\Pi$.
\end{definition}
By Lemma \ref{lem:negchamber}, $\theta$ is the highest root of 
$\Sigma$ or $\ck\theta$ is the highest root of $\ck\Sigma$.
We shall prove

\begin{theorem}\label{thm:elemext}
Let $\Phi$ be a finite crystallographic root system in the Euclidean
space $V$, and let $\Psi\subsetneq\Phi$ be a root subsystem.
Then there is a root subsystem $\wh\Psi$ with 
$\Psi\subsetneq\wh\Psi\subseteq\Phi$, such that either 
$\Psi$ is parabolic in $\wh\Psi$ or $\wh\Psi$ is an elementary
extension of $\Psi$.
\end{theorem}
\begin{proof}
Let $\Gamma$ be a simple system for $\Psi$, with corresponding 
positive system $\Psi_+\subseteq\Psi$. Denote the Weyl group
of $\Psi$ by $W'$ and let 
$\CC=\{v\in V\mid \mpair{\gamma,v}\geq 0\text{ for }\gamma\in\Gamma\}$
be the 
fundamental chamber for the action of $W'$ corresponding to 
the simple system $\Gamma$.

Since $\Psi\neq\Phi$, there is a root $\alpha\in\Phi\setminus\Psi$. 
There is a $W'$-transform $w\alpha\in -\CC$ ($w\in W'$), and since clearly
$w\alpha\not\in\Psi$, we may assume that $\alpha\in-\CC$. It follows that
$\mpair{\alpha,\gamma}\leq 0$ for all $\gamma\in\Gamma$.

Let $\wt\Gamma=\Gamma\cup\{\alpha\}$. If $\wt\Gamma$ is linearly
independent, then by Lemma \ref{lem:finsimple}, $\wt\Gamma$ is
a simple subsystem of $\Phi$, and taking $\wh\Psi$ to be the 
root system corresponding to $\wt\Gamma$, clearly $\Psi$ is
a parabolic subsystem of $\wh\Psi$.

Suppose then that $\wt\Gamma$ is linearly dependent. As with any set of roots,
$\wt\Gamma$ has a unique partition 
\be\label{eq:partgamma}
\wt\Gamma=\Gamma_1\cup\Gamma_2\cup\dots\cup\Gamma_k,
\ee
where the $\Gamma_i$ are mutually orthogonal (i.e. if $\gamma_i\in\Gamma_i$
and $i\neq j$ then $\mpair{\gamma_i,\gamma_j}=0$), and cannot be further
decomposed as in (\ref{eq:partgamma}).

Evidently $\alpha$ lies in precisely one of the sets $\Gamma_i$, and we may
suppose $\alpha\in\Gamma_1$. Now the matrix of pairwise
inner products of the elements of $\Gamma_1$ satisfy the conditions
of \cite[Ch. V, \S 3, Lemma 5]{Bour}, hence there are positive real
numbers $c_\gamma$ such that $\sum_{\gamma\in\Gamma_1}c_\gamma\gamma=0$,
and the $c_\gamma$ are uniquely determined, up to a common multiple,
and for all $\gamma$, $c_\gamma\neq 0$. It follows from Lemma \ref{lem:finsimple}
that every proper subset of $\Gamma_1$ is a simple subsystem of $\Phi$.

Now the matrix $(\mpair{\gamma,\ck\delta})_{\gamma,\delta\in\Gamma_1}$
is an indecomposable generalised Cartan matrix of affine type in the sense
of \cite[\S 15.3]{Ca:2005}. These have been completely classified 
(see, e.g. {\it op.~cit. \S\S 15.3, 17.1}), and the following statement
is an easy consequence of this classification (for an alternative proof using affine 
Weyl groups, and independent of the classification,  see Corollary \ref{cor:indecnip}(b)).

There is an element $\theta\in\Gamma_1$ such that 
if $\Gamma_1'=\Gamma_1\setminus\{\theta\}$, then 
\begin{conds}
\item $\Gamma_1'$ is a simple system for an indecomposable root system 
$\Psi_1'$, and
\item $\theta\in\Psi_1'$.
\end{conds}

By Lemma \ref{lem:negchamber},
$\theta$ is the negative of either the highest root of $\Psi_1$
or of the coroot of the highest root of $\ck\Psi_1$.

Let $\wh\Psi=\Psi_1'\cup\Psi_2\cup\dots\cup\Psi_k$, where $\Psi_1'$ is 
as above, and for $i=2,\dots,d$,
$\Psi_i$ is the root subsystem with simple system $\Gamma_i$.
Then $\wh \Psi$ is evidently an elementary extension
of $\Psi$ with the desired properties.
\end{proof}

\begin{corollary}\label{cor:weylsubsyst}
Let $\Phi$ be a finite crystallographic root system in the Euclidean
space $V$. The Dynkin diagram of any root subsystem of $\Phi$ is obtained
by a sequence of operations of the following type on  Dynkin diagrams, 
beginning with the Dynkin diagram of $\Phi$:
\begin{enumerate}
\item Delete one or more nodes.
\item Replace some indecomposable component $\Sigma$ of the Dynkin diagram by
the completed Dynkin diagram of $\Sigma$ or $\ck\Sigma$, and delete at 
least one node from the resulting affine component. \end{enumerate}
\end{corollary}
\begin{proof}
If $\Psi$ is a parabolic subsystem of $\Phi$, its Dynkin diagram is obtained
from that of $\Phi$ by deleting nodes.
By Theorem 1, we need only relate the Dynkin diagram of a subsystem $\Psi$
of $\Phi$ to that of $\Phi$ when $\Phi$ is an elementary extension of $\Psi$.
In that case, taking into account Lemma \ref{lem:negchamber},
the result follows from the nature of an elementary extension,
given that a root system and its dual have the same (unoriented) Dynkin diagram.
\end{proof}

It is well-known that any parabolic subsystem of $\Phi$ is closed in $\Phi$.
Also, it follows from the definitions and the results of Borel and de Siebenthal \cite{K} 
that if $\Phi$ is an elementary extension of $\Psi$, then either $\Psi$ is closed in $\Phi$ 
or $\ck \Psi $ is closed in $\ck \Phi$. Hence we have the following: 
\begin{corollary}\label{cor1} 
 If $\Psi$ is a maximal proper root subsystem of $\Phi$, 
then either $\Psi$ is closed in $\Phi$, or $\ck \Psi$ is closed in $\ck \Phi$. \end{corollary}

\section{Affine Weyl groups}\label{s:aff}

\subsection{Background}
We introduce some notation and  present a summary of the necessary background on affine Weyl  groups, 
most of which may be found in \cite[Ch. VI, \S 2]{Bour} or \cite[Ch 6--7]{Kac}.

Let $Z$ be any additive subgroup of $\real$ (typically $Z=\Int$ or $Z=\real$ in applications).  
Let $\Psi$ be any root system in $V$.  We define $Q_{Z}(\Psi)$ to be the additive subgroup of $V$ 
generated by elements
  $z\ck \alpha$ for $z\in Z$ and $\alpha\in \Psi$. We define $P_{Z}(\Psi)$  as the additive subgroup
\[P_{Z}({\Psi}):=\mset{v\in \real \Psi\mid\mpair{v,\ck\alpha}\in Z \text{ for all $\alpha\in \Psi$}}\]
of $V$. Let $\Gamma$ be any simple system of $\Psi$. Define the corresponding  fundamental coweights 
$\omega_{\gamma}=\omega_{\gamma}(\Gamma)\in \real \Psi$ of $\Psi$,  
for $\gamma\in \Gamma$, by $\mpair{\gamma,\omega_{\gamma'}}=\delta_{\gamma,\gamma'}$. 
Then $Q_{Z}(\Psi)=\oplus_{\alpha\in \Gamma}Z\ck \alpha\cong Z^{\vert \Gamma\vert}$ 
and $P_{Z}(\Psi)=\oplus_{\alpha\in \Gamma}Z\omega_{\alpha}\cong Z^{\vert \Gamma\vert}$

The cases $P(\Psi):=P_{\Int}(\Psi)$ and  $Q(\Psi):=Q_{\Int}(\Psi)$ are of special significance; 
these are the the coweight lattice and the coroot lattice of $\Psi$
(under the  natural identifications of the dual space of  $\real\Psi$ with $\real \Psi$
induced by $\mpair{-,-}$). We shall sometimes write $P(\Gamma)$ for $P(\Psi)$ and $Q(\Gamma)$ for $Q(\Psi)$.

       The index $[P(\Gamma):Q(\Gamma)]$ is called the index of connection
       (\cite[Ch VI, \S 1,no. 9]{Bour}) of $\Psi$
       (or $\Gamma$) and will be denoted as $f_{\Psi}$ or $f_{\Gamma}$.
       It is known (\cite[Ch VI, \S 1, Ex. 7]{Bour}) that 
       \be \label{eq:cartdet} f_{\Gamma}=\det ((\mpair{\alpha,\ck \beta})_{\alpha,\beta\in \Gamma})\ee  
       is equal to the determinant of the Cartan matrix of $\Psi$.

\subsection{} For the remainder of \S\ref{s:aff}, we take $\Phi$ to be a reduced finite 
crystallographic root system in $V$, as in 
\S \ref{intro}.  For simplicity, assume $\Phi$ spans $V$. Write $P=P(\Phi)$ and $Q=Q(\Phi)$.
 For $v\in V$, denote by $\tau_v:V\to V$ the translation
$x\mapsto x+v$. For any subgroup $B$ of $V$, let $\tau_{B}:=\mset{\tau_{\alpha}\mid\alpha\in B}$ 
denote the group of translations of $V$ by elements of $B$. 

 For $\alpha\in\Phi$ and $m\in\Z$,
define the affine reflection
\be\label{eq:affineref}
s_{\alpha,m}(v)=v-(\mpair{\alpha,v}-m)\ck\alpha=s_{\alpha}(v)+m\ck\alpha.\ee
We let $H_{\alpha,m}:=\mset{v\in V\mid s_{\alpha,m}(v)=v}=\mset{v\in V\mid \mpair{v,\alpha}=m}$ 
denote the reflecting hyperplane of $s_{\alpha,m}$.

\begin{definition}\label{def:affinew}
The {\it affine Weyl group} $W^a$ associated to $\Phi$ is the group generated by
$\mset{s_{\alpha,m}\mid \alpha\in\Phi,\;\;m\in\Z}$. The {\it extended affine Weyl group} $\wt {W}^{a}$
is the group generated by  $W\cup \tau_{P}$. The alcoves of $W^{a}$ are the connected components
(for the standard topology of $V$)  of $V\setminus \cup_{\alpha\in \Phi,n\in \Int}H_{\alpha,n}$.
A {\it special point} of $W^{a}$ is a point of $V$ which for each $\alpha\in \Phi$, 
lies in some reflecting hyperplane $H_{\alpha,m}$ ($m\in \Int$).
\end{definition}

\begin{proposition}\label{prop:affinewequiv} 
\begin{num}
\item The group $W^a$ is the semidirect product $W\ltimes \tau_{Q}$ of $W$ 
and $\tau_{Q}$, while $\wt{W}^{a}$
is the semidirect product $W\ltimes \tau_{P}$ of $W$ and $\tau_{P}$.
Moreover $W^{a}$ is a normal subgroup of 
$\wt W^{a}$  (i.e. $ W^{a}\trianglelefteq \wt{W}^{a}$)
such that $ \wt W^{a}/ {W}^{a}\simeq P/Q$.
\item Let $\Pi$ be a simple system of $\Phi$. Then
$W^a$ is a Coxeter group, with Coxeter generators 
$S^{a}:=\set{s_\alpha\mid\alpha\in\Pi}\cup\set{s_{\theta,1}\mid \theta\in H}$, 
where $H$ is the set
of highest roots of the indecomposable components of $\Phi$ (with respect to their simple systems
contained in $\Pi$).
\item Let $A:=\mset{v\in V\mid \mpair{v,\alpha}\geq  0\text{ \rm for all  $\alpha\in \Pi$},\quad
\mpair{v,\theta}\leq 1\text{ \rm for all  $\theta\in H$}}$. Then $A$ is the closure of the 
fundamental alcove $A^{0}$ for $(W^{a},S^{a})$ on $V$. In particular, $A$ is the 
closure of  $A^{0}$, $A^{0}$ is the interior of $A$, and $S^{a}$ is the set of reflections 
in the  walls $\set{H_{\alpha,0}\mid \alpha\in \Pi}\cup\set{H_{\theta,1}\mid \theta\in H}$ of  $A$. 
\item Each closed alcove contains at least one special point of $W^{a}$. The set of all special points 
coincides with the coweight lattice of $\Phi$.
\item The indecomposable components of $W^{a}$ are the affine Weyl groups associated 
to the indecomposable components of $\Phi$, and the group
 $\wt W^{a}$ is the direct product of the extended affine Weyl groups 
 of the indecomposable components of $\Phi$.
\end{num}
\end{proposition}
\begin{remark} See \cite[Ch V, \S 1--3]{Bour} for more details on alcoves, walls, special points etc. For
 indecomposable $\Phi$, the above facts  can be found in \cite{Bour}. For decomposable $\Phi$, 
 the statements are either in \cite{Bour} or follow readily  from the results there about
decomposable affine Weyl groups. It is well known that any closed alcove is a fundamental 
domain for the action of $W^{a}$ on $V$. In this paper, by a fundamental domain  for a 
group acting on a set, we shall simply mean a set of orbit representatives, though our 
fundamental domains will usually satisfy additional properties such as connectedness.\end{remark}

\subsection{Linear realisation}\label{ss:linreal}
The groups $W^a$ and $\wt W^a$ appear above as groups 
of affine transformations of $V$. They may also by viewed as groups of linear transformations
of a space $\wh V$ which contains $V$ as a subspace of codimension $1$.
In particular we shall see that there is an affine  root system in $\wh V$ on which both
$\wt W^{a}$ and $W^{a}$ act naturally.

Let $\wh V=V\oplus\real \delta$, where $\delta$ is an indeterminate.
The inner product $\mpair{-,-}$ on $V$ has a unique extension to a 
symmetric bilinear form on $\wh V$
which is positive semidefinite, and has radical
equal to the subspace $\real \delta$. This extension to $\wh V$ 
will also be denoted by $\mpair{-,-}$. For any non-isotropic element
$v\in \wh V$, let $\wh s_{v}\colon \wh V\rightarrow \wh V$ denote the $\real$-linear  
map $x\mapsto x-\mpair{x,\ck v}v$ where $\ck v=\frac{2}{\mpair{v,v}} v$.
For any $v\in V$, let $t_{v}\colon \wh V\rightarrow \wh V$ denote the $\real$-linear map 
$x\mapsto x+\mpair{x,v}\delta$.  

Let $\Pi$ be a simple system of  $\Phi(\subset V\subset\wh V)$ with corresponding set $\Phi_{+}$ 
of positive roots.   
\begin{definition}\label{def:phia} Define the following subsets of $\wh V$:
\be\label{eq:phia}
\begin{aligned}
&\Phi^a:=\Phi+\Z \delta\\
&\wh \Pi:=\Pi\cup\set{-\theta+\delta\mid \theta\in H}, \text{ and}\\ 
&\Phi^a_{+}:=
\mset{\alpha+n\delta\mid \alpha\in \Phi,n\in \Int_{\geq 0}, n>0 \text{ \rm if $\alpha\in -\Phi_{+}$}}.\\
\end{aligned}
\ee
\end{definition}
\begin{proposition}\label{prop:afflinear}
\begin{num} 
\item The group $\wt W^{a}$ (resp., $W^{a}$) is isomorphic to the
group generated by $W$ and $t_{\gamma}$ for $\gamma\in P$ (resp., $\gamma\in Q$). 
Under this isomorphism, $t_{\gamma}$ is mapped to $\tau_{-\gamma}$
and for $\alpha\in \Phi$ and $m\in \Z$, $s_{\alpha,m}$ is mapped to
$\wh s_{\alpha-m\delta}=\wh s_{m\delta-\alpha}$.
\item We have:
 $\Phi^a$ is $\wt W^{a}$-stable, $W^{a}=\mpair{\wh s_{\alpha}\mid \alpha\in \Phi^a}$,
 $\Phi^a=W^{a}\wh\Pi=\Phi_{+}\dot\cup -\Phi_{+}$,  
 and $\Phi_{+}=\Phi\cap \real_{\geq 0}\wh \Pi$. 
 Moreover the set $S^{a}$ of $\text{\rm Proposition \ref{prop:affinewequiv}}$ is mapped to 
 $\mset{\wh s_{\alpha}\mid \alpha\in \wh \Pi}$.
\end{num}
\end{proposition}
\begin{proof}  
Let $W^{a}_{1}$ (resp., $\wt W^{a}_{1}$)  denote the group of linear 
endomorphisms of $\wh V$  defined in (a). 

The proof will use the following elementary formulae concerning the action of these groups on $\wh V$,
which are listed here for future reference.

\begin{equation} \label{eq1}\wh s_{\alpha}t_{\gamma}=t_{\wh s_{\alpha}(\gamma)}\wh s_{\alpha},\quad\text{for 
$\alpha\in \Phi^a$,  $\gamma\in \wh V$}
  \end{equation}

\begin{equation} \wh s_{\alpha+n\delta}=t_{n\ck \alpha} \wh s_{\alpha},\quad 
\text{for $\alpha\in \Phi$,  $n\in \Int$} \label{eq2}  \end{equation}

\begin{equation} \wh s_{\alpha+n\delta}\wh s_{\alpha+m\delta}=t_{(n-m)\ck \alpha},\quad
\text{for  $ \alpha\in \Phi$,  $ n,m\in \Int$} \label{eq3}       \end{equation}

\begin{equation}
\wh s_{\alpha+n\delta}(\beta+m\delta)=s_{\alpha}(\beta)+(m-n\mpair{\beta,\ck \alpha})\delta
\quad\text{for $\alpha,\beta\in \Phi$ and $n,m\in \Int$} \label{eq4}      \end{equation}

\begin{equation} \label{eq5}  wt_{\gamma}(\alpha+m\delta)=w(\alpha)+(m+\mpair{\alpha,\gamma})
\delta\quad \text{for $w\in W$, $\gamma\in V$,  $\alpha\in \Phi$, $m\in \Int$.} \end{equation}

Consider  the contragredient action of $\wt W_{1}^{a}$ and $W_{1}^{a}$ on the 
dual space \[\wh V^{*}:=\text{Hom}_{\real}(\wh V,\real),\]
defined by $wf=f\circ w^{-1}$ for $f\in \wh V^{*}$ and $w\in \wh W^{a}$. 
Define $d\in \wh V^{*}$ by $d(v+c\delta)=c$ for $v\in V$, $c\in \real$.
The affine  map $i\colon V\rightarrow \wh V^{*}$ defined by
$i(v)(x)=\mpair{v,x}+d(x)$ for $v\in V$ and $x\in \wh V$, is injective, and
has image equal to the
affine hyperplane $E=\mset{f\in \wh V^{*}\mid f(\delta)=1}$ of $\wh V^{*}$.
This is $\wt W_{1}^{a}$-stable
since $\wt W_{1}^{a}$ fixes $\delta$. Since $E$ spans $\wh V^*$ as real vector space, 
restriction to $E$ defines an isomorphism 
of $\wt W_{1}^{a}$ (and hence of $W_{1}^{a}$) with a group of affine transformations of $E$.
But from the definitions, one readily checks that 

\begin{equation}  t_{\gamma} i=i\tau_{-\gamma}\quad\text{ for $\gamma\in V$}\label{eq6}  
\end{equation}
and
\begin{equation} \wh s_{\alpha}i=is_{\alpha}\quad\text{for $\alpha\in \Phi$}\label{eq7}   
 \end{equation}
The first assertion of (a) follows by noting that by \eqref{eq6},\eqref{eq7}, $i:V\to E$ 
is an isomorphism 
of $\wt W^a$-modules.
We observe also that under the above map, the element
$s_{\alpha,m}=\tau_{m\ck \alpha}s_{\alpha}$
of $ W^{a}$ is mapped to 
$t_{-m \ck\alpha}\wh s_{\alpha}=\wh s_{m\delta-\alpha}$ in $\wh W^{a}_{1}$.
This completes the proof of (a).

It is not difficult to prove (b) using the above formulae and well known 
facts about root systems of Weyl groups. Instead, we simply observe that the proof of (b)
reduces easily to the special case in which $\Phi$ is indecomposable, in which case 
$\Phi^a$ may be identified with the system of real roots of 
an untwisted affine Kac-Moody Lie algebra \cite[Prop. 6.3(a)]{Kac}, and 
given that $\Phi^a$ is $\wh W^{a}$-stable, the claims of (b) are well known
(see \cite[Ch 6--7]{Kac}). To see that $\Phi^a$ is $\wh W^{a}$-stable, 
it is clear that $\Phi^a$ is stable under $W^{a}$, and to prove stability
under $\wh W^a$, we apply \eqref{eq5} 
with $\gamma$ in the coweight lattice and note that $\mpair{\alpha,\gamma}\in \Z$.\end{proof}

\subsection{}\label{genextaff}  Let $R$ be a fixed  subgroup of $P$ containing $Q$. It is well-known that $R$ is 
$W$-stable, so one may form the semidirect product $\wt W:=W\ltimes \tau_{R}$ which satisfies 
$W^{a}\subseteq \wt W\subseteq \wt W^{a}$ and $[\wt W:W^{a}]=[R:Q]$. Many of our results for 
$\wt W^{a}$ apply equally to any of the subgroups $\wt W$.

For $v\in V\subset \wh V$, we have 
$\wh s_v=s_{v,0}\oplus \id_{\real\delta}=s_{v}\oplus \id_{\real\delta}$.
Henceforth, we shall therefore write $s_{\alpha}$ 
instead of $\wh s_{\alpha}$ for any $\alpha\in \Phi$. We shall regard
$\wt W^{a}$, $W^{a}$ and $\wt W$ as groups of linear transformations of $\wh V$, or as 
groups of affine transformations of $V$, as convenience dictates.

Let $\Delta\subseteq\wh V\setminus \real \delta$.
The {\it (indecomposable) components} of $\Delta$ are the inclusion-minimal 
non-empty subsets $\Delta'$ of $\Delta$ which are orthogonal to their complement in $\Delta$. 
Say that $\Delta$ is {\it indecomposable} if it has only one component. 
If $\Delta$ is a simple subsystem or a root subsystem of $\Phi^a$, these 
coincide with the usual notions of indecomposable components and indecomposability.

\subsection{Linear versus affine} \label{ss:standrootsyst}
This  subsection and the next 
two describe relationships between $\Phi^a$
and root systems of affine Weyl groups, and the realisation
of $W^a$ as a group of affine transformations in a smaller space.
It will turn out that comparison of the two realisations of affine Weyl groups
as linear groups and as affine linear groups can be effectively exploited.

Recall (\cite[Ch IV, \S 1.9]{Bour}) that for any  set $I$, an $I$-indexed Coxeter matrix
is any symmetric matrix of the form $M=(m_{i,j})_{i,i\in I}$, where 
$m_{i j}\in \Z_{>0}\amalg\{\infty\}$, and $m_{i j}=1$
if and only if $i =j$. Let $M$ be such a  Coxeter matrix.
The Coxeter system $(W_{M},S_{M})$ with Coxeter matrix $M$ has a 
standard faithful representation, the ``Tits representation'',
as a reflection group of isometries of a space $U$ with basis 
$e_{i}$ for $i\in I$. This representation is 
described in \cite[Ch V, \S 4]{Bour} and \cite[5.3--5.4]{Hum}. Denote by 
$B_{M}$ the matrix $B_{M}=(-\cos\frac{\pi}{m_{i,j}})_{i,j\in I}$.
The space $U$ has a symmetric $\real$-valued bilinear form $(-,-)$
satisfying $((e_{i},e_{j}))_{i,j\in I}=B_{M}$. The simple reflections $S_{M}$ are the 
orthogonal reflections  $r_{i}$ on $U$ whose
fixed hyperplanes are the $e_{i}^\perp$ ($i\in I$); $W_{M}$ 
is then the group of linear transformations of $U$ generated by $S_{M}$. There is an associated 
root system $\Phi_{M}$ for $(W_{M},S_{M})$ defined by  $\Phi_{M}=W_{M}\Pi_{M}$ where 
$\Pi_{M}=\set{e_{i}\mid i\in I}$.

\subsection{} Now take $\wh V=V\oplus \real\delta$ as above,
and note that with respect to the form $\mpair{-,-}$, the set of 
isotropic elements of $\wh V$ is $\real\delta$.
For any non-isotropic  $\alpha\in \wh V$, we set  
$\dot\alpha:= \frac{1}{\Vert \alpha\Vert}\alpha$  
where ${\Vert \alpha\Vert}:=\mpair{\alpha,\alpha}^{1/2}$. 
  Define the map $\iota\colon \wh V\setminus \real\delta 
  \rightarrow \wh V\setminus \real\delta$ by $\iota(\alpha)=\dot \alpha$.
For any subset $\Gamma$ of $\wh V\setminus \real\delta$, we set  
$\ck \Gamma:=\mset{\ck \alpha\mid \alpha\in \Gamma}$ and $\dot \Gamma:=\iota(\Gamma)=
\mset{\dot \alpha\mid \alpha\in \Gamma}$.

\begin{remark} If $M$ is the Coxeter matrix of $(W^{a},S^{a})$ and $W^a$ is  
indecomposable, then there is an identification $U=\wh V$ (as real vector 
spaces with symmetric bilinear forms) such that 
$\Phi_{M}=\dot{\wh \Psi}$, $ \Pi_{M}=\dot{\wh \Pi}$  
and $(W^{a},S^{a})=(W_{M},S_{M})$.  
This is not true if $W^a$ is decomposable, since $\wh \Pi$
is then linearly dependent, 
and $\wh V$ is obtained from $U$ by taking 
the quotient by an appropriate subspace of 
the radical of $U$. The advantage of the root system $\Phi^a$  
is that its  root subsystems may be treated in a similar way
to $\Phi^a$ itself, since even for  
indecomposable $\Phi$, root subsystems of $\Phi^a$ need not have 
linearly independent simple roots.\end{remark}

\subsection{}\label{ss:based root datum} 

It is evident from the proof of Proposition \ref{prop:afflinear}, 
that the indecomposable components of the root system $\Phi^a$ 
may be identified with the real root systems of untwisted affine Kac-Moody Lie 
algebras associated to the components of $\Phi$ \cite{Kac}.  
If $\Phi$ is not indecomposable, then $\wh \Pi$ is not linearly independent. 
However, $\wh \Pi$  is positively independent in the following sense:  
a subset  $\Sigma$ of a $\real$-vector space is said to be {\it positively independent} if
$\sum_{\gamma\in \Sigma}c_{\gamma}\gamma=0$ with all $c_{\gamma}\geq 0$ and 
almost all $c_{\gamma}=0$  implies that all $c_{\gamma}=0$.   
It follows  that     $(\wh V,\mpair{-,-},\dot{ \wh \Pi})$ is a based root datum
with  root system $\dot{ \Phi^a}$ in the sense of \cite{Semi} with  
positive roots $\dot{\Phi^a}_{+}=\iota(\Phi^a_{+})$ and associated 
Coxeter system $( W^{a}, S^{a})$.
This is enough to ensure that standard facts about root systems 
of Coxeter groups apply to $\Phi^a$; see \cite{Semi} for properties 
of based root systems  and \cite{DyRig}, \cite{Mo} for discussion of  
root systems closely related to those considered here.

   \subsection{}  The fact proved below will play an important role in this paper.
Let  $M=(m_{\alpha\beta})_{\alpha,\beta\in\Gamma}$ be Coxeter matrix where  $\Gamma$ is a finite set,  $B_M:=\left(-\cos\frac{\pi}{m_{\alpha\beta}}\right)_{\alpha,\beta\in\Gamma}$ be the matrix  of the invariant form on the space of the Tits representation of the corresponding Coxeter
group.

\begin{proposition}\label{prop:indec} 
 Suppose that $\Gamma\subseteq V\setminus\set{0}$ is indecomposable and
linearly dependent, and that the matrix $B:=(\mpair{\dot\alpha,\dot\beta})_{{\alpha,\beta\in \Gamma}}$ 
is equal to $B_{M}$ for some $\Gamma$-indexed Coxeter matrix $M$
(cf. \S \ref{ss:standrootsyst}).
Then there exists $\theta\in \Gamma$ and scalars $c_{\gamma}\in \real_{>0}$
for $\gamma\in \Gamma$  with the following property:  the set
$\Lambda:=\mset{c_{\gamma}\gamma\mid \gamma\in \Gamma\setminus \set{\theta}}$ is 
the set of simple roots of an indecomposable, reduced  crystallographic root system 
$\Psi$  and $-c_{\theta}\theta$ 
is the highest root of $\Psi$ with respect to $\Lambda$.
Moreover, the affine Weyl group of $\Psi$ is naturally
isomorphic to $W_{M}$ and the set $\set{c_{\gamma}}_{\gamma\in \Gamma}$ 
is uniquely determined up to multiplication by a positive scalar.  
\end{proposition}

\begin{proof} 
Consider the Coxeter matrix 
$M=(m_{\alpha,\beta})_{\alpha,\beta\in \Gamma}$ and the 
associated reflection representation of $(W_{M},S_{M})$ in $U$, 
a $\real$-vector space with basis $e_{\alpha}$ for $\alpha\in \Gamma$, 
as described in \S\ref{ss:standrootsyst}. Set $\Delta:=\mset{e_{\alpha}\mid \alpha\in \Gamma}$. We have by 
construction that 
\[ \label{innprodmat1} ((e_{\alpha},e_{\beta}))_{\alpha,\beta\in \Gamma}=B. \] 

If $\alpha\neq\beta$, then  $(e_{\alpha},e_{\beta})=-\cos\frac{\pi}{m_{\alpha\beta}}\leq 0$.
It follows from \cite[Ch V, \S 3, Lemme 5]{Bour}, the subspace 
$\real \Gamma$ has dimension $\vert \Gamma\vert -1$ and 
there is a linear relation
$\sum_{\alpha\in \Gamma}c_{\alpha}\alpha=0$ which is unique up to multiplication
by a non-zero scalar; moreover this relation may be uniquely chosen so that $c_{\alpha}>0$ 
and $\sum_{\alpha\in \Gamma}c_{\alpha}=1$.
Hence in particular our assumptions imply that the matrix $B$ is positive semidefinite of corank $1$.

   The linear map $L\colon U\rightarrow V$ defined by $e_{\alpha}\mapsto \dot \alpha$
for $\alpha\in \Gamma$   satisfies $\mpair{L(u),L(u')}=(u,u')$.
The form $(-,-)$ on $U$ is positive semidefinite and  clearly  $y:=\sum_{\alpha\in \Delta} c_{\alpha}e_{\alpha}\in U$
spans   its  radical. Following \cite[Ch V, \S 4.9]{Bour}, let $U^{*}$ denote the dual space of $U$ and 
 $F$ denote  the affine subspace $F=\mset{f\in U^{*}\mid f(y)=1}$. This is a coset of the subspace
 $T:=\mset{f\in U^{*}\mid f(y)=0}$ of $U^*$, so $T$ acts simply transitively by translation on $F$.
 Clearly $T$ is canonically isomorphic to $(U/ \real y)^*$. The form  $\mpair{-,-}$ on $U$ induces 
 a  natural positive definite inner product on the quotient  $U/\real y$, which may thus be naturally
 identified with $T:=(U/\real y)^{*}$, which therefore has a natural induced structure of 
 finite-dimensional real inner product space. Note that since $\mpair{y,e_\alpha}=0$ for all 
 $\alpha\in\Gamma$, $W_{M}$ fixes $y$, and hence acts on $F$ as a group of affine isometries. 
 Let $C:=\mset{f\in U^{*}\mid  f(\Delta)\subseteq\real_{\geq 0}}$ denote the (closed)
 standard fundamental chamber 
 for $(W_{M},S_{M})$ acting on $U^{*}$; it is a closed simplicial cone, with walls
$\mset{f\in U^{*}\mid  f(e_{\alpha})=0}$ for $\alpha\in \Gamma$.
 The group $W$ acts faithfully by restriction as a group of affine transformations  of $F$, and we  
 shall therefore identify $W$ with a  group of affine transformations of $F$ (cf. \cite[Ch V, \S 4,  
 Prop. 10]{Bour}). Then $W$ is a faithful, effective, indecomposable group generated by  
 affine reflections. The intersection $A:=C\cap F$ is a closed  fundamental alcove for $W$ acting on 
 $F$. Note that $A$ is a closed  simplex in $F$. Its walls are the intersections of $F$ with the walls 
 of $C$ and hence are also indexed 
 naturally by $\Gamma$. Consider the  unit  normals (in $T$) to the
 walls of $C$, in the direction of $C$ from the wall.  The matrix of inner products of these 
 normals is $(e_{\alpha},e_{\beta})_{\alpha,\beta\in \Gamma}$ by \cite[Ch V, \S 4,  Prop. 10]{Bour}.

By \cite[Ch V, \S 3,  Prop. 12]{Bour}, we may chose a special point $O$ of $W_{M}$ on $F$ such 
that $O\in A$. Taking origin to be $0=O$, the affine space $F$ acquires the structure of a
real Euclidean  space, which is mapped isomorphically to $T$ by translation by $-O$.
Denote the inner product on $F$ as $(-\mid -)$.
It follows from \cite[Ch VI, \S 2, Prop. 8]{Bour}, that $W_{M}$ acting on $F$ is the 
affine Weyl group of a reduced crystallographic essential root system
$\Psi$ in $F$, as in Proposition \ref{prop:affinewequiv}, with Weyl group $W_{0}$. Since $W_{M}$ is   indecomposable, $\Psi$ 
is  indecomposable. Every closed alcove of $W_{M}$ on $F$ is contained in a closed chamber of $W_{0}$ on $F$. 
 Choose a set of simple roots $\Lambda$ for $\Psi$ such that the corresponding closed  fundamental chamber
$C'$ for  $W_{0}$ contains $A$. Then $A$ is the unique alcove of $W_{M}$ in $F$ which is contained  
in $C'$ and which contains $0$ in its closure.  Hence $A$ is the closed  fundamental alcove for 
$(W_{M},S_{M})$. Let $\omega_{0}$ denote the highest root of $\Psi$ with respect to $\Lambda$.
Then we have by Proposition \ref{prop:affinewequiv}  that 
\[A=\mset{v\in F\mid (v\mid\beta)\geq 0 \text{ \rm for all $\beta\in \Gamma$ and 
$(v\mid \omega_{0})\leq 1$} }.\]  It follows that the unit  normals to the walls 
of $ A$  are $\set{\dot \beta\mid \beta\in \Lambda'}$ where $\Lambda':= \Lambda\cup\set{-\omega_{0}}$.

This shows that the matrix $((\dot\alpha\mid\dot \beta))_{\alpha,\beta\in \Lambda'}$ 
coincides (up to permutations of its rows and columns)
with the matrix $(\mpair{\dot\alpha,\dot\beta})_{{\alpha,\beta\in \Gamma}}$.
Since $\Lambda'$ spans $F$ and $\Gamma$ spans $\real \Gamma$,
this implies  that there is a unique  identification of $F=\real \Gamma$ as inner product spaces 
such that $\mset{\dot \alpha\mid\alpha\in \Gamma}=\set{\dot \beta\mid \beta\in \Lambda'}$. 
Let $ \theta \in \Gamma$ with $\dot  \theta =\dot  \theta _{0}$. 
We may write $\Lambda'=\mset{c_{\alpha}\alpha\mid \alpha\in \Gamma}$ 
for uniquely determined $c_{\alpha}\in \real_{>0}$.  All 
assertions of the Proposition are now clear.  \end{proof}

\subsection{} The following Corollary provides a classification-free 
proof of the claim at the end of  the proof of Theorem \ref{thm:elemext}.
\begin{corollary} \label{cor:indecnip}
Let $\Phi$ be a crystallographic root system, and let
$\Gamma\subseteq \Phi$ be an indecomposable subset, which is linearly dependent over $\real$ 
and satisfies $\mpair{\alpha,\beta}\leq 0$ for all distinct  $\alpha,\beta\in \Gamma$.
Then
\begin{num}\item Let $B:=(\mpair{\dot\alpha,\dot\beta})_{{\alpha,\beta\in \Gamma}}$
and $M=(m_{\alpha,\beta})_{\alpha,\beta\in \Pi}$ where $m_{\alpha,\beta}$ denote 
the order of $s_{\alpha}s_{\beta}$ if $\alpha\neq- \beta$, and $m_{\alpha,\beta}:=\infty$ if $\alpha=-\beta$.
Then $M$ is the Coxeter matrix of an indecomposable affine Weyl group, and $B=B_{M}$.
\item  Let $W':=\mpair{s_{\alpha}\mid \alpha\in \Gamma}$.
There is an element $\theta\in \Gamma$ such that $\Gamma':=\Gamma\setminus\set{\theta}$ 
is a simple system for the root system 
$\Phi'\subseteq \Phi$ of $W'$, and $\theta\in \Phi'$. \end{num}
\end{corollary}

\begin{proof}
  Clearly $M$ is a Coxeter matrix. Since $\Phi$ 
is crystallographic, rank two considerations \cite[Ch VI, \S 1.3]{Bour} 
imply that  $\mpair{\dot \alpha,\dot \beta}=-\cos\frac{\pi}{m_{\alpha,\beta}}$ for all $\alpha,\beta\in \Gamma$. Hence 
$B=B_{M}$, and thus we are in the situation of Proposition \ref{prop:indec}. It follows 
from that Proposition that $M$ is the Coxeter matrix of an indecomposable affine Weyl
group, whose root system we denote by $\Psi$. Let $\theta$ be as in the conclusion of
Proposition \ref{prop:indec}.

Since $\Gamma':=\Gamma\setminus\set{\theta}$ is linearly independent, it is a simple system 
of some root subsystem $\Phi'$ of $\Phi$, by  Lemma \ref{lem:finsimple}, and 
evidently $\Phi':=W'\Gamma'$ where $W':=\mpair{s_{\alpha}\mid \alpha\in \Gamma'}$.
Moreover since some multiple $c_\theta \theta\in\Phi'$, the reflection $s_\theta\in W'$
and hence $W'=\mpair{s_{\alpha}\mid \alpha\in \Gamma}$.
It follows also that there is some $\alpha\in \Phi'$ with $s_{\alpha}=s_{ \theta }$.
But $\alpha\in \Phi'$ and $s_{\alpha}=s_{ \theta }$ implies $\alpha=\pm  \theta $, 
whence $\theta \in \Phi'$. 
 \end{proof}

\section{Subsystems  of affine root systems I}\label{s1}
In this section, we  classify the reflection subgroups of $W^{a}$ (a) up to equality
and (b) up to isomorphism. The classification (b) was conjectured (in equivalent form) and 
partly proved by Coxeter \cite{Cox}, and proved  in \cite{DyRef}. The proof here via the 
stronger result (a) is along similar lines, but is much   simpler and  independent of the 
classification of affine Weyl groups. The classification in (a) is obtained in conjunction
with a classification all simple subsystems of $\Phi^a$ (up to equality); the
reflection subgroups are in bijective correspondence with those simple subsystems 
which are contained  in a fixed positive subsystem of the ambient affine root system.

      \subsection{} \label{ss1.3} 

Let $\Phi$ be a crystallographic root system in $V$ and $\Phi^a$ the corresponding
affine root system (in $\wh V$, see Definition \ref{def:phia}). For 
non-isotropic $v\in \wh V$, we henceforth write $s_v$ for the orthogonal reflection 
in $v^\perp$, so that for $x\in\wh V$, $s_v(x)=x-\mpair{x,\frac{2v}{\mpair{v,v}}}v$.
Note that this transformation $s_v$ was denoted $\wh s_v$ in \S\ref{ss:linreal}. 
A root subsystem $\Sigma$ of $\Phi^a$ is a subset of  $\Phi^a$ with the property 
that $s_{\alpha}(\beta)\in \Sigma$ for all $\alpha,\beta\in \Sigma$. 
The root subsystems of $\Phi^a$ are in natural bijective  correspondence with 
the reflection subgroups of $W^{a}$ by the map which to a reflection subgroup  $W'$ 
associates the set of all roots $\alpha$ in $\Phi^a$ such that $s_{\alpha}\in W'$; 
the inverse map associates to a root subsystem $\Sigma$ the subgroup 
$\mpair{s_{\alpha}\mid \alpha\in \Sigma}$ (we use here simple consequences 
of properties of reflection subgroups of Coxeter systems from
\cite{DyRef}).

The notion of a simple subsystem of $\Phi^a$ is more delicate than 
the corresponding notion for $\Phi$
(see \cite{HowTay}, \cite{DyRig}).
Consider the following set of real numbers.
\[\COS:=\mset{-\cos\frac{\pi}{m}\mid m\in \Nat_{\geq 2}}\cup(-\infty, -1]\subseteq \real.\]
Let $\Sigma$ be a root subsystem of $\Phi^a$ with corresponding reflection subgroup $W'$.
A subset $\Gamma$  of $\Phi^a$ is called  a {\it simple system} for $\Sigma$   if 
$W'=\mpair{s_{\alpha}\mid \alpha\in \Gamma}$, 
$\Gamma$  is positively independent and $\mpair{\dot \alpha,\dot \beta}\in -\COS$ for 
all distinct $\alpha,\beta\in \Gamma$. Then $\Sigma=W'\Lambda$. 

\begin{remark}
In \cite{DyRig} the notion of {\it abstract simple system} is defined
by replacing positive independence of $\Gamma$ by the weaker condition that each 
component of $\Gamma$ be positively independent.
In this work we shall not make use of this concept.
\end{remark}

For any simple system 
$\Lambda$ of $\Sigma$, $S_{\Lambda}:=\mset{s_{\alpha}\mid \alpha\in \Lambda}$ is a 
set of Coxeter generators of $W'$.
If $\Sigma\subseteq \Phi$, its simple systems in the above sense coincide 
with those in the sense of \S \ref{intro}.

Note that any subset of $\Phi^a_{+}$ is positively independent.
It therefore follows from the general theory of reflection subgroups of Coxeter systems 
(\cite{DyRef},\cite{Semi}) that $\Sigma$ has a unique simple system $\Lambda$ 
contained in $\Phi_{+}$; we call this the canonical simple system of $\Sigma$ (or of the 
corresponding reflection subgroup).
We write $\Lambda=\Delta_{W'}$. Then $S_{\Lambda}$ is the `canonical' set of Coxeter generators 
of $W'$ (in the sense of \cite{DyRef}), with respect to  the simple reflections $S^{a}$ of $W^{a}$. 
If $\Sigma$ is irreducible, then  any two  simple  systems of $\Sigma$ are in the same 
$W'\times\set{\pm 1}$-orbit, as follows from \cite{Mo},
from \ref{ss:based root datum} and \cite{HowTay}, or from \cite{DyRig}.

  If $\Delta$ is a simple system of some root subsystem $\Sigma$ of   
  $\Phi^a$, we refer to $\Delta$ as a {\it simple subsystem} of $\Phi^a$.
  If  $\Delta\subseteq \Phi$, we write
  $W_{\Delta}:=\mpair{s_{\alpha}\mid \alpha\in \Delta}$ for the subgroup of $W$ 
  generated by reflections corresponding to elements of $\Phi$ and   
  $\Phi_{\Delta}=W_{\Delta}\Delta:=\mset{\alpha\in \Phi\mid s_{\alpha}\in W_{\Delta}}$ 
  for the corresponding root subsystem of $\Phi$.

\subsection{} \label{ss1.4}  
 Let   $\pi$ denote   the projection $\wh V\rightarrow V$  with  $\delta\mapsto 0$. 
 Note that  
 \begin{equation}\label{rootproj}
 \pi(\Phi^a)\subseteq \Phi,\qquad \mpair{\alpha,\beta}=
 \mpair{\pi(\alpha),\pi(\beta)} \text{ \rm for  $\alpha,\beta\in \Phi^a$}.\end{equation}
 The next result follows immediately from this  by well known properties \cite[Ch VI, \S 2, no. 3]{Bour}  
 of rank two root subsystems of $\Phi$.

 \begin{lemma}\label{lem1}
 \begin{num} \item A subset $\Delta$ of  $\Phi$ (resp. $\Phi^a$) is a 
 simple subsystem of $\Phi$   if and only if $\Delta$ is positively independent,
and $\mpair{\alpha,\beta}\leq 0$ for all distinct $\alpha,\beta\in \Delta$.
\item  A subset $\Delta$
  of $\Phi_{+}$ (resp., $\Phi^a_{+}$) is the canonical simple system of  some reflection subgroup 
  $W'$ of $W$ (resp. $W^a $) if and only if  $\mpair{\alpha,\beta}\leq 0$ for all distinct
  $\alpha,\beta\in \Delta$.\end{num} \end{lemma}
  \begin{remark} More generally, this Lemma applies mutatis mutandis to  
  real root systems of Kac-Moody Lie algebras e.g. see  \cite[5.5, Theorem 1]{Mo}.\end{remark}

 \subsection{Negative product subsets}   \label{ss1.5}  
\begin{definition}\label{def:npsubset}
We say that a subset $\Delta$ of $\wh V$ is a  \nip  (negative  product  subset) 
   if $\mpair{\alpha,\beta}\leq 0$ for all distinct $\alpha,\beta\in \Delta$.
\end{definition}
   \begin{lemma}\label{nipproj}
   \begin{num}\item    $\Delta\subseteq \Phi^a$ is a \nip if and only if 
   the restriction of $\pi$ to $\Delta$ is injective and  $\pi(\Delta)$ is a \nip  of $\Phi$.  
   \item Given a \nip 
   $\Delta'$ of $\Phi$, the \nips $\Delta$ of $\Phi^a$ with 
   $\pi(\Delta)=\Delta'$ are the sets of the form 
   $\Delta=\mset{\alpha+f(\alpha)\delta\mid \alpha\in \Delta'}$ 
   where $f\colon \Delta'\rightarrow \Int$ is an arbitrary function.
  \item For $\Delta$ and $f$ as in $\text{\rm (b)}$, 
  We have $\Delta\subseteq \Phi^a_{+}$ if and only if $f(\alpha)\geq 0$ 
  for $\alpha\in \Phi_{+}$ and $f(\alpha)>0$ for $\alpha\in -\Phi_{+}$.
  \end{num}\end{lemma} 
 \begin{proof} This follows directly from \eqref{rootproj}. \end{proof}

\subsection{} As an immediate Corollary of  Lemmas  \ref{lem1} and \ref{nipproj},  
we obtain a   parameterisation of  the reflection subgroups  of $W^a $ 
(or equivalently, root subsystems of $\Phi^a$) in terms of \nips of $\Phi$. 

For any additive  subgroup $Z$ of $\real$, let  $\RR(\Phi,Z)$ denote the set of all 
pairs $(\Gamma,f)$ where $\Gamma$ is a \nip of $\Phi$
and $f\colon \Gamma\rightarrow Z\cap \real_{\geq 0}$ is a function such that  
$f(\alpha)>0$ if $\alpha\in -\Phi_{+}\cap \Gamma$.

\begin{theorem}\label{thm1}  
The reflection subgroups of the affine 
Weyl group $W^a $ are in natural bijective correspondence with $\RR(\Phi,\Int)$. 
The correspondence attaches to $(\Gamma,f)\in \RR(\Phi,\Int)$ the unique  reflection  
subgroup $W'$ of  $W^a $ with $\Delta_{W'}=\mset{\alpha+f(\alpha)\delta\colon \alpha\in \Gamma}$.
\end{theorem}
We denote the reflection subgroup $W'$ of $W^{a}$ attached to the pair 
$(\Gamma,f)\in \RR(\Phi,\Int)$ of the Theorem by $W^{a}(\Gamma,f)$,   
its corresponding root system in $\Phi^a$ by $\Phi^a(\Gamma,f)$, and 
its canonical simple system by $\Delta(\Gamma,f)\subseteq \Phi_{+}$.
A particular case is  $W^{a}=W^{a}(\Gamma_{0},f_{0})$, where
$(\Gamma_{0},f_{0})\in \RR(\Phi,\Int)$  is defined by
$\Gamma_{0}=\pi(\wh \Pi)$ and \begin{equation}f_{0}(\alpha):=\begin{cases} 0,&\text{if $\alpha\in \Phi_{+}$}\\
1,&\text{if $\alpha\in -\Phi_{+}$}\end{cases}\end{equation}

 \subsection{Classification of \nips} \label{ss1.6}   
 The  \nips of $\Phi$ are described in the following Lemma.
     \begin{lemma}\label{lem2}
Let $\Gamma'$ be a simple subsystem of $\Phi$ with  indecomposable  components $\Gamma'_{i}$ 
for $i=1,\ldots, n$. For each $i$, let $\Gamma_{i}$ denote either $\Gamma'_{i}$ or 
$\Gamma'_{i}\cup\set{\alpha_{i}}$ where either  $-\alpha_{i}$ is the highest root 
of  $\Phi_{\Gamma_{i}'}$ or $-\ck \alpha_{i}$ is the highest root of $\ck  \Phi_{\Gamma_{i}'}$ (or both).
Then $\Gamma:=\cup_{i}\Gamma_{i}$ is a \nip  of $\Phi$ whose components are the $\Gamma_{i}$.
Moreover any \nip  $\Gamma$ of $\Phi$ is obtained in this way
from some simple subsystem $\Gamma'$ of $\Phi$.
\end{lemma}  
\begin{proof} Well known properties of highest roots \cite[Ch VI,\S 1. 8]{Bour}
imply that $\Gamma$ as in the Lemma is a \nip of $\Phi$.  The converse follows by applying 
Proposition \ref{prop:indec}  to the indecomposable components $\Gamma_{i}$ of $\Gamma$ 
and using Lemma \ref{lem:negchamber}.
\end{proof} 

\begin{remark}\label{rem3} 
Suppose  $\Gamma$ is a \nip of $\Phi$. 
The components $\Gamma_{i}$ are uniquely determined  by $\Gamma$, but $\Gamma_{i}$
need not uniquely determine $\alpha_{i}$ and 
$\Gamma_{i}'$ as above  if $\Gamma_{i}'\neq \Gamma_{i}$.
However, $W_{\Gamma}=W_{\Gamma'}$, and any two choices of $\Gamma'$ 
(for fixed $\Gamma$) are in the same $W_{\Gamma}$ orbit.  
See also Remark \ref{rem:nipstab}.
\end{remark}

\subsection{} \label{s1.7}  
We  now  obtain the following known description (\cite{DyRef}) of the reflection subgroups 
of $W^a $, up to isomorphism as Coxeter systems.

\begin{theorem} \label{thm2} 
Consider a root subsystem $\Psi$ of $\Phi$ 
with  indecomposable components 
$\Psi_{i}$ for $i=1,\ldots, n$. For each $i$, let $W_{i}$ denote either the Weyl group 
of $\Psi_{i}$, the affine Weyl group corresponding to $\Psi_{i}$ or the affine Weyl 
group corresponding to $\ck \Psi_{i}$. Each $W_{i}$ has a standard realisation as a Coxeter 
system, with simple reflections determined up to conjugacy. 

The affine Weyl group $W^a $ has a reflection 
subgroup isomorphic to $W':=W_{1}\times \ldots \times W_{n}$
as a Coxeter system. Conversely, any reflection subgroup of $W^a $ 
is isomorphic to a Coxeter system $W'$ arising in this way.   \end{theorem} 

\begin{proof} Using Lemma \ref{lem2}, the isomorphism type of the reflection subgroup 
attached to 
$(\Gamma,f)$ in Theorem  \ref{thm1} may be described as follows. Suppose $\Gamma$ 
is obtained as in Lemma \ref{lem2} from the simple subsystem  $\Gamma'$ of $\Phi$, and let 
$\Gamma_{i},\Gamma'_{i}$ be as the sets of roots defined there.
Then $W'$ has  indecomposable components $W'_{i}$ with simple roots 
$\Delta_{i}=\mset{\alpha+f(\alpha)\delta\colon \alpha\in \Gamma_{i}}$, 
for $i=1,\ldots, n$.
If $\Gamma_{i}=\Gamma'_{i}$, then $W'_{i}$ is isomorphic to the Weyl 
group of $\Phi_{\Gamma'_{i}}$.  If $\Gamma_{i}=\Gamma'_{i}\cup\set{\alpha_{i}}$ 
where $-\alpha_{i}$ is the highest root of $\Phi_{\Gamma'_{i}}$,  then $W'_{i}$ is 
isomorphic to the affine  Weyl group of $\Phi_{\Gamma'_{i}}$.  Finally,  
if $\Gamma_{i}=\Gamma'_{i}\cup\set{\alpha_{i}}$ where $-\ck \alpha_{i}$ is 
the highest root of $\ck \Phi_{\Gamma'_{i}}$,  then $W'_{i}$ is isomorphic to 
the affine  Weyl group of $\ck\Phi_{\Gamma'_{i}}$. 
\end{proof}

\subsection{} \label{ss4.3}
The next Lemma  gives necessary and sufficient  
conditions for  a \nip $\Delta\subseteq \Phi^a$ to be positively independent in terms 
of its  parameterisation  in Lemma \ref{nipproj}(b).

\begin{lemma} \label{lem4} 
Let $\Gamma$ be a \nip  of  $\Phi$ and 
$f\colon \Gamma\rightarrow \Int$ be a function.  Let $\Gamma_{i}$ for 
$i=1,\ldots,n$ denote the components of  $\Gamma$.
Let $J$ denote the set of indices $i$ with $1\leq i\leq n$ and $\Gamma_{i}$ 
linearly dependent.
For each $i\in J$ there are unique 
integers $c_{\gamma}\in \Nat_{> 0}$  for $\gamma\in\Gamma_i$
such that  $\sum_{\gamma\in \Gamma_{i}}c_{\gamma}\gamma=0$ 
and  $\gcd(\mset{c_{\gamma}\mid \gamma\in \Gamma_{i}})=1$; we set  
$d_{i}:=\sum_{\gamma\in \Gamma_{i}}   c_{\gamma}f_{\gamma}$. Let $\Delta:=\mset{\gamma+f(\gamma)\delta
\mid \gamma\in \Gamma}$
\begin{num}\item $\Delta$ is positively independent if and only if  either $d_{i}>0$ 
for all $i\in J$ or $d_{i}<0$ for all $i\in J$.
\item each component of $\Delta$ is positively independent if and only if $d_{i}\neq 0$ for all $i\in J$.
\end{num} 

\end{lemma}
\begin{proof} Each $\Gamma_{i}$ is either linearly independent or there 
is a  (real) linear relation among its elements, which is
unique up to a scalar  by \cite[Ch V, \S 3, Lemma 5]{Bour}. 
Since $\Gamma$ is contained in the lattice  $\Int \Phi$, the coefficients of these linear
relations can
be taken in $\rat_{>0}$ by loc.  cit., and then by rescaling,  the existence and uniqueness  of the $c_{\gamma}$  
for $\gamma\in \Gamma_{i}$, $i\in J$  satisfying the conditions  of the Lemma follows.

For any set of scalars $a_{\gamma}\in \real_{\geq 0}$ ($\gamma\in \Gamma$), 
we have $\sum_{\gamma\in \Gamma}a_{\gamma}(\gamma+f(\gamma)\delta)=0$ if
and only if
$\sum_{\gamma}a_{\gamma}\gamma=0$ and $\sum_{\gamma}a_{\gamma }f(\gamma)=0$.
Now by \cite{Bour},   $\sum_{\gamma}a_{\gamma}\gamma=0$ if and
only if there are non-negative  
scalars $b_{i}$ for $i\in J$,  such that $a_{\gamma}=b_{i}c_{\gamma}$ 
if $\gamma\in \Gamma_{i}$ for some $i\in J$, and 
$a_{\gamma}=0$  otherwise. For such scalars $a_{\gamma}$, we have
$\sum_{\gamma}a_{\gamma}f(\gamma)=\sum_{i\in J}b_{i}d_{i}$. Hence 
$\Delta$ is positively independent precisely if,
for all non-negative scalars $b_{i}$ ($i\in J$) which are not all zero,
$\sum_{i\in J}b_{i}d_{i}\neq 0$. The assertion (a) follows immediately from this.
The  components of $\Delta$ are $\Delta_{i}:= \mset{\alpha+f(\alpha)\delta\colon \alpha\in \Gamma_{i}}$ 
for $i=1,\ldots, n$, so (b) follows from (a) applied to each  $\Gamma_{i}$.
\end{proof}

\begin{definition}\label{def:compatible}
If the condition in $\text{\rm (a)}$ (resp. $\text{\rm (b)}$) holds, 
we say that $(\Gamma,f)$ is compatible (resp. strongly compatible).
\end{definition}

\subsection{} \label{ss4.4} 
 The following parameterisation of 
simple subsystems of $\Phi^a$  follows immediately from Lemma \ref{lem4} and the definitions, 
in the same way as Theorem \ref{thm1}.
\begin{proposition}\label{prop3} 
The simple 
subsystems  of the root system $\Phi^a$ of $W^a $  are in natural bijective 
correspondence with pairs $(\Gamma,f)$ where $\Gamma$ is a \nip of $\Phi$ and
$f\colon \Gamma\rightarrow \Int$ is a function such that $(\Gamma,f)$ is 
compatible.
 The correspondence attaches to $(\Gamma,f)$ the simple  
 subsystem $\mset{\alpha+f(\alpha)\delta\colon \alpha\in \Gamma}$ of  $\Phi^a$.\end{proposition}
 \begin{remark} We note that rank $1$ simple subsystems of $\Phi^{a}$ correspond 
 naturally to roots in $\Phi^{a}$.  Theorem \ref{thm1} and Proposition \ref{prop3}
 directly generalize the descriptions (\eqref{eq:phia}, Proposition \ref{prop:afflinear})  of the 
 positive roots and  roots of $\Phi^a$ in terms of those of $\Phi$.\end{remark}

\subsection{Short and long roots} \label{ss:rootlengthnotation}  
For any non-isotropic 
$v\in \wh V$, let  $k_{v}=\mpair{v,v}\in \real_{>0}$.  If $\Psi$ is a subset of some indecomposable 
crystallographic root system in $V$, then $\mset{k_{v}\mid v\in \Psi}$ has either $1$ or $2$ elements.
   If this set  has two elements, we denote them by $k_{\Psi,+}$ and $k_{\Psi,-}$
   where $k_{\Psi,{+}}>k_{\Psi,-}$. Otherwise, we define
   $k_{\Psi,+}=k_{\Psi,-}:=k_{v}$ for any $v\in \Psi$.
   Set $k_{\Psi}:=\frac{k_{\Psi,+}}{k_{\Psi,-}}$. It is known from \cite{Bour}
   that $k_{\Psi}\in \set{1,2,3}$. 
   Write $\Psi_{\text{\rm short}}:=\mset{v\in \Psi\mid \mpair{v,v}=k_{\Psi,-}}$ and 
    $\Psi_{\text{\rm long}}:=\mset{v\in \Psi\mid \mpair{v,v}=k_{\Psi,+}}$.
      
Using the above notation and that of Lemma \ref{lem2}, we now define
$k_{\beta,\Gamma}\in \set{1,2,3}$ for any \nip $\Gamma$ and element $\beta\in \Gamma$.  
There is a unique component $\Gamma_i$ in which $\beta$ lies.
If  $\Gamma_{i}=\Gamma'_{i}$, we set  $k_{\beta,\Gamma}:=1$.
Otherwise,  $\Gamma_{i}\neq \Gamma'_{i}$,  and we set
\begin{equation*}k_{\beta,\Gamma}:=\begin{cases} 1,&\text{ if   
$k_{\alpha_{i}}=k_{\Gamma_{i},+}$}\\
 \frac{k_{\beta}}{k_{\alpha_{i}}}\in \set{1,k_{\Gamma_{i}}}&\text{ if  
 $k_{\alpha_i}=k_{\Gamma_{i},-}$.}\end{cases}\end{equation*}

\subsection{}  The technical lemma below  describes the  root coefficients in certain 
root subsystems of $\Phi^a$.
Let $\Gamma$ be an  indecomposable \nip of $\Phi$. Let $\Gamma':=\Gamma$ 
if $\Gamma$ is linearly independent.
Otherwise, write  $\Gamma=\Gamma'\cup \set{-\alpha}$
where $\Gamma'$ is an indecomposable simple subsystem of $\Phi$  and 
$\mpair{\alpha,\Gamma'}\subseteq \real_{\geq 0}$. Let $\Sigma$ be the 
root subsystem of $\Phi$ with simple system $\Gamma'$.
   
\begin{lemma}  \label{lem:affrootcoeff}  
Maintain the above notation.
If $\Gamma=\Gamma'$, set $\Delta':=\Gamma$; otherwise, set  
$\Delta':=\Gamma'\cup\set{\omega}\subseteq \Phi^a$ where $\omega:=\delta-\alpha$.
Let $W':=\mpair{s_{\gamma}\mid \gamma\in \Delta'}$. Then $W'$ is  
a reflection subgroup of $W^a $ having $\Delta'$ as a simple system. 
Let $\Psi':=W'\Delta$ denote the indecomposable root subsystem of $\Phi^a$ 
corresponding to $\Delta'$.

\begin{num}\item Define coefficients $a_{\beta,\gamma}$ 
for $\beta\in \Sigma$ and $\gamma\in \Gamma'$, by
$\beta=\sum_{\gamma\in \Gamma'}a_{\beta,\gamma}\gamma$.
If $\Gamma$ is linearly dependent, then  
\[\beta+n\delta=n\omega+\sum_{\gamma\in \Gamma'}(a_{\beta,\gamma}+na_{\alpha,\gamma})\gamma\]
for $\beta\in \Sigma$ and $n\in \Int$.
\item  If $\Gamma$ is linearly independent, 
then $\Psi'=\Sigma$.
\item If $\Gamma$ is linearly dependent then 
$\Psi'=\mset{\beta+nk_{\beta,\Gamma}\delta\mid \beta\in \Sigma, n\in \Int}$.\end{num}
\end{lemma}\begin{proof} Parts (a) and (b) are  trivial. In (c),  the case  
$\Gamma\neq \Gamma'$ with  $k_{\alpha}=k_{\Gamma,+}$ is a special case of 
Proposition \ref{prop:afflinear} since then all $k_{\beta,\Gamma}=1$.
The remaining case is that in which   $\Gamma\neq \Gamma'$ and $k_{\alpha}=k_{\Gamma,-}$.
Set $b_{\beta}=\frac{2}{k_{\beta}}$ for $\beta\in \Sigma$, so $\ck \beta=b_{\beta}\beta$.
We have  $k_{\ck\alpha}=k_{\ck \Gamma,+}$. 
Consider  $\ck{ \Delta'}=\ck{\Gamma'}\cup\set{b_{\alpha}\delta-\ck \alpha}$.
Here, $\delta':=b_{\alpha}\delta$ plays the analogous role for 
$\ck{\Delta'}$ to that played by  $\delta$ for $\Delta'$.
By the case of (c) previously considered,  the 
 root system  $\ck{ \Psi'}$ corresponding to $\ck {\Delta'}$ is
 \[ \ck{ \Psi'}=\mset{ \beta +n\delta'\mid  \beta \in \ck \Sigma, n\in \Int}=
 \mset{\ck \beta +nb_{\alpha}\delta\mid  \beta \in  \Sigma, n\in \Int}.\]
 Hence \[\Psi'=  \mset{\ck{(\ck  \beta +nb_{\alpha}\delta)}\mid  \beta \in  
 \Sigma, n\in \Int}=\mset{ \beta +n\frac{b_{\alpha}}{b_{ \beta }}\delta\mid  \beta 
 \in  \Sigma, n\in \Int}\] which is  as required since 
 $\frac{b_{\alpha}}{b_{\beta}}=\frac{k_{\beta}}{k_{\alpha}}=k_{\beta,\Gamma}$.
\end{proof}

\subsection{} \label{ss:nipnot}  
Let $\Gamma$ be a \nip of $\Phi$ and let $f\colon \Gamma\rightarrow \Int$ such 
that $(\Gamma,f)$ is compatible.
We assume that $\Gamma$ arises as in Lemma \ref{lem2} from a simple system 
$\Gamma'\subseteq \Gamma$ and use the  notation   
$\Gamma_{i}$, $\Gamma'_{i}$, $\alpha_{i}$ introduced there. 

For each $i$ with $\Gamma'_{i}\neq \Gamma_{i}$, there is a unique linear relation
$\sum_{\gamma\in \Gamma_{i}} c_{\gamma}\gamma=0$ with $c_{\alpha_{i}}=1$, and
in this relation, we have 
$c_{\gamma}\in \Nat_{>0}$ for all $\alpha\in \Gamma_{i}$.

Let $\Sigma=\Phi_{\Gamma'}$ and $\Sigma_{i}=\Phi_{\Gamma'_{i}}$; thus $\Sigma$ is a root
system with indecomposable components $\Sigma_{i}$, of which
$\Gamma'$ is a simple system, and $\Gamma'\subseteq\Gamma\subseteq \Sigma$.
As in Lemma \ref{lem:affrootcoeff},
for each $\beta\in \Sigma$, write $\beta=\sum_{\gamma\in \Gamma'}a_{\beta,\gamma}\gamma$.
Then $a_{\beta,\gamma}\in \Int$,
$a_{\beta,\gamma}=0$ unless $\beta\in \Sigma_{i}$ and $\gamma\in \Gamma_{i}'$ 
for some $i$, and for fixed $\beta\in \Sigma$, either all $a_{\beta,\gamma}$  
are non-negative or they are all non-positive.

By Proposition \ref{prop3}, 
$\Delta:=\mset{\gamma+f(\gamma)\delta\mid \gamma\in \Gamma}$ is a 
simple subsystem of $\Phi^a$. Let $W'':=\mpair{s_{\gamma}\mid \gamma \in \Delta}$ 
be the associated Coxeter group, and $\Psi:=W''\Delta$ be the corresponding root system.

Define  $r_{\beta}\in \Int$ for all $\beta\in \Sigma$ by 
$r_{\beta}:=\sum_{\gamma\in \Gamma'}a_{\beta,\gamma}f(\gamma)=-r_{-\beta}$.
We also define  $K_{i}\in \Int$ for $i=1,\ldots, n$ as follows.
If $\Gamma_{i}=\Gamma'_{i}$, we set $K_{i}=0$. Otherwise,
set  $K_{i}:= \sum_{\gamma\in \Gamma_{i}}c_{\gamma}f(\gamma)=
f(\alpha_{i})+r_{-\alpha_{i}}$.

\begin{proposition} \label{prop:niprootsyst}
Let $\Gamma$ be a \nip of $\Phi$ and $f\colon \Gamma\rightarrow \Int$
a function such that $(\Gamma,f)$ is compatible. Then the root system 
$\Psi$ of which
$\Delta:=\mset{\alpha+f(\alpha)\delta\mid \alpha\in \Gamma}$ 
is a simple system is given by
\[\Psi=\mset{\beta+(r_{\beta}+mK_{i}k_{\beta,\Gamma})\delta\mid 
i=1,\ldots, n,\beta\in \Sigma_{i},m\in \Int}.\]
\end{proposition}
\begin{remark} The above  result applies in particular  in the case that
  $\Delta\subseteq \Phi^a_{+}$, when $\Psi=\Phi^a(\Gamma,f)$. In that case,  
  we have $r_{\beta}\geq 0$ for 
all $\beta\in \Sigma$ which are positive with respect to $\Gamma'$. If, further, 
 $\Gamma_{i}'\neq \Gamma_{i}$, then using \cite[Ch VI,  \S 1, Prop. 25]{Bour} 
 for $\Sigma$ and $\ck \Sigma $, we find that $ 0\leq r_{\beta}\leq K_{i}k_{\beta,\Gamma}$  for all $\beta\in \Sigma_{i}$
 which are positive with respect to  $\Gamma'_{i}$.
\end{remark}
\begin{proof} The proof easily reduces to the case where $\Gamma$ is indecomposable,
which we henceforth assume.
Let $\Delta'$, $\Psi'$ be as in Lemma \ref{lem:affrootcoeff}.    
There is an isometry $i\colon \real \Delta'\rightarrow \real \Delta$, mapping the basis $\Delta'$ 
of the left side to the basis $\Delta$ of the right side; more precisely,  
$\gamma\mapsto\gamma+f(\gamma)\delta$ for $\gamma\in \Gamma'$, and, if $\Gamma'\neq \Gamma$, 
 $\delta-\alpha\mapsto  f(-\alpha)\delta-\alpha$.  Clearly,  
 this isometry induces an isomorphism $W'\rightarrow W''$ of Coxeter systems and  maps $\Psi'$ bijectively to $\Psi$. 
 From the expressions for elements of $\Psi'$ as linear combinations of the simple 
 roots $\Delta'$ in Lemma \ref{lem:affrootcoeff}, 
 we obtain corresponding expressions for  elements of  
 $\Psi$ as linear combinations of  simple roots $\Delta$. One readily checks that 
 expanding the latter linear combinations  gives the desired result, using the 
 definitions of $r_{\beta}$, $K_{i}$ and $k_{\beta,\Gamma}$.\end{proof}

\subsection{Weyl data, alcoves and compatible pairs} 
For the remainder of this section, fix  $(\Gamma,f)\in \RR(\Phi,\Int)$. 
We shall show how to explicitly describe in terms of the data $(\Gamma,f)$
  the fundamental alcove of
 $W':W(\Gamma,f)$,  the volume of an alcove,  the index $[W^{a}:W']$,
and the minimal length (say, left) coset representatives 
of $W'$ in $W^{a}$. To describe these results, 
we maintain all  the notation   from the last subsection.

 It is clear that $\Phi^a(\Gamma,f)$ is finite  (equivalently, $W^{a}(\Gamma,f)$  
 is finite) if and only if $\Gamma=\Gamma'$. Also, $W^{a}(\Gamma,f)$ has only affine type  
 components (i.e.  no finite components) if and only if   $\Gamma_{i}\neq \Gamma'_{i}$ for all $i$.
Finally, $W^{a}(\Gamma,f)$ is of finite index in $W^{a}$ if and only if its translation 
subgroup is of finite index in that of $W^{a}$, which is the case if and only if 
$\Gamma_{i}\neq \Gamma'_{i}$ for all $i$ and $\real \Phi=\real \Gamma'$.      

\subsection{}
 The fundamental coweights  
$\omega_{\alpha}:=\omega_{\alpha}(\Gamma')$ for $\alpha\in\Gamma'$ are defined
as the basis of $\R\Gamma'$ which is dual to $\Gamma'$ with respect to $\mpair{-,-}$. 
Recall  that $f_{\Gamma'}$ denotes  the index of connection of $\Phi_{\Gamma'}$
($=[\sum_{\gamma\in\Gamma'}\Z\omega_\gamma:\sum_{\gamma\in\Gamma'}\Z\gamma]$) and let  $r_{i}:=\vert \Gamma'_{i}\vert$ denote the rank of $\Phi_{\Gamma'_{i}}$.

\begin{lemma}\label{lem:alcove}
\begin{num}
\item The closed fundamental alcove for $W^{a}(\Gamma,f)$ on $V$  is 
\[C=C(\Gamma,f):=\mset{v\in V\mid\mpair{v,\gamma}+f(\gamma)\geq 0\text{ for all $\gamma\in \Gamma$}}.\]
\item $C=\perp_{i=0}^{n} C_{{i}} $ (orthogonal direct sum of subsets) where $C_{0}$ is the subspace  $C_{0}:=\mset{v\in V\mid \mpair{v,\Gamma}=0}$ and for $i=1,\ldots, n$, 
\[C_{i}=\mset{v\in \real\Gamma'_{i} \mid\mpair{v,\gamma}+f(\gamma)\geq 0\text{ for all $\gamma\in \Gamma_{i}$}}.\]
\item  For $i=1,\ldots, n$ let $v_{i,0}:=-\sum_{\alpha\in \Gamma'_{i}}f(\alpha)\omega_\alpha$.
If $\Gamma'_{i}=\Gamma_{i}$, then $C_{i}$ is the simplicial cone
$v_{i,0}+\sum_{\alpha\in \Gamma_{i}}\real_{\geq 0} \omega_{\alpha}$ with vertex $v_{i_{0}}$,
while if $\Gamma'_{i}\neq \Gamma_{i}$,  then $C_{i}$ is the simplex with vertices 
$v_{i,0}$ and $v_{i,0}+\frac{K_{i}}{c_{\alpha}}\omega_{\alpha}$ for 
$\alpha\in \Gamma'_{i}$. 
\end{num}\end{lemma} 
\begin{proof}
We refer to \cite[Proposition 3.12]{Kac}, \cite[Ch V, \S 4, no. 4--6]{Bour},
\cite[5.6]{Mo} and  \cite{Semi} for basic facts concerning
fundamental chambers and Tits cones used below. The closed fundamental 
chamber for $W^{a}(\Gamma,f)$ acting on $\wh V^{*}$  is
\bee 
\wh C:=\mset{\phi \in \wh V^{*}\mid \phi(\alpha+f(\alpha)\delta))\geq 0
\text{ \rm for all $\alpha\in \Gamma$}}
\eee  
and the corresponding Tits cone is 
\bee \wh X:=\mset{\phi\in \wh V^{*}\mid \phi(\beta)<0\text{ \rm for at most finitely many $\beta\in \Phi^{a}(\Gamma,f)\cap \Phi_{+}$}}.
\eee
 Clearly, $\wh X$   contains the Tits cone  of $W^{a}$ acting on $\wh V^{*}$, which is 
\bee \mset{\phi \in \wh V^{*}\mid\phi(\beta)<0\text{ \rm for at most finitely 
many $\beta\in  \Phi^a_{+}$}}\supseteq i(V),
\eee where $i:V\to\wh V^*$ is as in 
the proof of Proposition \ref{prop:afflinear}.   
Since $\wh C$ is a fundamental domain for $W^{a}(\Gamma,f)$ on $\wh X$ and
$i(V)$ is $W$-stable, it follows that $C$ is a fundamental domain 
for $W^{a}(\Gamma,f)$ acting on $V$. The reflections in the walls of $C$ 
correspond to the reflections in the walls of 
$\wh C$, which are the canonical Coxeter generators of $W^{a}(\Gamma,f)$, and (a) follows.

Part (b) follows readily from (a) and the definitions.

 For (c),  note first that $C_{i}$ is a simplicial cone or simplex, depending on
 whether $\Gamma'_{i}=\Gamma_{i}$ or $\Gamma'_{i}\neq \Gamma_{i}$. In either case,   
 one verifies first that $v_{i,0}$ 
lies on all walls of $C_{i}$ of the form 
$\mpair{v\in \real \Gamma_{i}\mid \mpair{v,\gamma}+f(\gamma)=0}$ for $\gamma\in \Gamma'_{i}$. 
If $\Gamma'_{i}=\Gamma_{i}$ then the elements $\omega_{\alpha}$ are parallel to the extreme rays  of $C_{i}$
and the assertion of (c) holds. If 
 $\Gamma'_{i}\neq \Gamma_{i}$, then one checks that 
\be \label{eq:vertexcheck} \mpair{v_{i,0},\alpha_{i}}+f(\alpha_{i})=K_{i}\ee 
so  that  $v_{i_{0}}\in C_{i}$ since $K_{i}>0$. Then for $\alpha \in \Gamma'_{i}$, 
one sees from \eqref{eq:vertexcheck}  that  the point
 $v_{i,0}+\frac{K_{i}}{c_{\alpha}}\omega_{\alpha}$  lies on all faces of $C_{i}$ except the
 face $\mpair{v\in \real \Gamma_{i}\mid \mpair{v,\alpha}+f(\alpha)=0}$, 
 so it is a vertex of $C_{i}$.
 The result (c) follows.
 \end{proof}
\subsection{Volume}  We maintain the notation of the above section and in 
addition,  consider Lebesgue measure $\mu$ on $V$, normalised so that the measure 
of an $m$-dimensional  hypercube in $V$ with sides of unit length 
(using the metric arising from $\mpair{-,-}$) is $1$, where $m=\text{\rm dim}(V)$.
Recall that $(\Gamma,f)$ is a fixed element of $\RR(\Phi,\Z)$, which defines a 
root subsystem of $\Phi^a$ as in Proposition \ref{prop3}. By Lemma \ref{lem:alcove}
the chamber $C$ is 
a fundamental domain for the action of $W(\Phi^a(\Gamma,f)$ on $V$.

\begin{corollary}\label{cor:alcvol}
\begin{num}
\item If the   measure $\mu(C)$ is finite, then 
\[\mu(C)=\sqrt{ \frac{1}{f_{\Gamma'}}\prod_{\beta\in \Gamma'}
\frac{2}{\mpair{\beta,\beta}}}\, \prod_{i}
\frac{K_{i}^{r_{i}}}{r_{i}!\prod_{\alpha\in \Gamma_{i}'}c_{\alpha}
}\]
\item $\mu(C)$ is finite if and only if the index $[W^{a}:W^{a}(\Gamma,f)]$ is finite.
\item  Suppose also $(\Gamma',f')\in \RR(\Phi,\Int)$ with
$W^{a}(\Gamma,f)\subseteq W^{a}(\Gamma',f')$  and that both  $\mu(C(\gamma,f))$
and $\mu(C(\Gamma',f'))$ are finite. Then    
\bee[W^{a}(\Gamma',f'):W^{a}(\Gamma,f)]=\frac{\mu(C(\Gamma,f))}{\mu(C(\Gamma',f'))}.\eee  
 \end{num}
\end{corollary}
\begin{remark} \label{rem7}
Since $W^{a}=W^{a}(\Gamma_{0},f_{0})$, the measure $\mu(A)$ of the 
fundamental alcove of $W^{a}$ is finite and may be computed as a special 
case of the formula in (a) \end{remark}
\begin{proof} If the measure $\mu(C)$ is finite, then $\real\Phi=\real\Gamma$ and for all $i$,
$\Gamma'_{i}\neq \Gamma_{i}$.  We assume we are in this situation. 
Let $\mu_{i}$ denote Lebesgue measure on $\real \Gamma_{i}$, 
normalised in a similar way to $\mu$. We have 
$\mu(C)=\prod_{i}\mu(C_{i})$ since the $C_{i}$ are pairwise orthogonal.
Consider the full-dimensional parallelepiped $P_{i}$ in $\real \Gamma_{i}$ 
with $0$ as one vertex and with the edges incident with 
$0$ given by endpoints $\omega_{\alpha}$ for $\alpha\in \Gamma_i'$.  
By  the change of variables formula in multiple integrals, its volume is 
\[\mu_{i}(P_{i})=\sqrt{\det(\mpair{\omega_{\alpha},\omega_{\beta}})_{\alpha,\beta\in \Gamma'_{i}} }\]
(this is also a well-known formula of Hadamard).
By expressing the elements of $\Gamma'$ and the corresponding coroots and fundamental  coweights in 
terms of an orthonormal basis of $\real \Gamma_{i}$ and using \eqref{eq:cartdet},  
one checks that the determinant under the square root sign is  equal to  
\[{\frac{1}{f_{\Gamma'_{i}}}\prod_{\beta\in \Gamma_{i}'}\frac{2}{\mpair{\beta,\beta}}}\]
On the other hand,  it is easy to see using Lemma \ref{lem:alcove}(c) and  the change 
of variable formula   that
\[\mu_{i}(C_{i})=\mu_{i}(P_{i})\prod_{i}\frac{K_{i}^{n_{i}}}{n_{i}!
\prod_{\alpha\in \Gamma_{i}'}c_{\alpha}
}.
\] 
Combining the above facts, using $f_{\Gamma'}=\prod_{i}f_{\Gamma'_{i}}$,  
gives (a). Parts (b) and (c) follow readily from \cite[Ch VI, \S 2, Lemma 1]{Bour}.
\end{proof}

\subsection{} From \cite{DyRef}, the set 
$\mset{w\in W^{a}\mid w\Delta(\Gamma,f)\subseteq \Phi^a_{+}}$ is known to be 
the set of minimal length left coset representatives
of $W^{a}(\Gamma,f)$ in   $W^{a}$ (where length is with
respect to the standard length function of $(W^{a},S^{a})$). 
The next result explicitly describes a set of left coset 
representatives of $W^{a}(\Gamma,f)$ in $\wt W=W\ltimes \tau_{R}$ (with $R$ as in \ref{genextaff}), reducing 
to the above set if $\wt W=W^{a}$
i.e. if $ R=Q$.
\begin{proposition} \label{prop:cosreps}
 Let $G:=\mset{w\in \wt W\mid w\Delta(\Gamma,f)\subseteq \Phi^a_{+}}$\begin{num} 
 \item For $w\in W$ and $\gamma\in  R$, we have $wt_{\gamma}\in G$ if
 and only if for all  $\alpha\in \Gamma$, 
\begin{equation*}\begin{cases}  \mpair{\alpha,\gamma}\geq -f(\alpha),&
\text{if  $w(\alpha)\in \Phi_{+}$}\\
\mpair{\alpha,\gamma}\geq 1-f(\alpha),&\text{if  $w(\alpha)\in \Phi_{-}$}
 \end{cases} \end{equation*}
 \item For $x\in \wt W$, we have $x\in G$ if and only if $x^{-1}(A)\subseteq C$.
 Further,  we have $C=\cup_{x\in G}\, x^{-1}(A)$. 
\item $G\cong \wt W/W^{a}(\Gamma,f)$ is a set of left coset representatives
of $W^{a}(\Gamma,f)$ in   $\wt W$. 
\item  $\vert G\vert =[W^{a}:W^{a}(\Gamma,f)] \cdot [R:Q]$.
\end{num}
  \end{proposition}
  \begin{proof}   Part (a) follows easily  from the definition 
  of $\Phi^a(\Gamma,f)$ and \eqref{eq5}. Note that 
   $W^{a}$ acts simply transitively on the set of its alcoves 
   and $\wt W=W\ltimes \tau_{R}$ acts on the set of alcoves. It follows 
   that $\wt W$ is the semi-direct product of its normal subgroup $W^{a}$
   with the subgroup 
   \[\mset{w\in \wt W\mid w(A)=A}=\mset{w\in \wt W\mid w(\wt \Pi)=\wt \Pi}\cong R/Q,
   \]and this reduces the proofs  of (b)--(d) to the case $\wt W=W^{a}$ where 
   the statements  are standard or trivial.\end{proof}


 \section{Subsystems  of affine root systems II} \label{s5}
 In the previous section, root subsystems of $\Phi^a$ were classified using the concept of 
 negative product (np) set, which arises naturally in considering simple subsystems of affine root
 systems. In this section we give an alternative classification in terms of collections
 of subsets of $\Z$.
 
   \subsection{}  For subsets $A$, $B$ of any abelian group,
   define $A\pm B:=\mset{a\pm b\mid a\in A,b\in B}$ and 
   $nA=\mset{na\mid a\in A}$ for $n\in \Int$.
   Note that $2A\neq A+A$ in general.

   Any subset $\Psi$  of $\Phi^{a}$ may be described as $\Psi:=\mset{\alpha+n\delta\mid {\alpha\in \Phi}, n\in Z_{\alpha}}$ where, 
   for each $\alpha\in \Phi$, $Z_{\alpha}$ is a subset of $\Int$.
   Given subsets $Z_{\alpha}$ for $\alpha\in\Phi$, the next lemma gives
   necessary and sufficient conditions on the $Z_{\alpha}$ in order that the
   corresponding set $\Psi$ be a root subsystem of $\Phi^a$. 
 
 \begin{lemma} \label{lem:Z} 
   Let
 $\set{Z_{\alpha}\mid \alpha\in \Phi}$ be a family of (possibly empty) 
 subsets of $\Int$. Then $\Psi:=\mset{\alpha+n\delta\mid {\alpha\in \Phi}, n\in Z_{\alpha}}$ 
 is a root subsystem of $\Phi^a$
 if and only if, for all $\alpha,\beta\in \Phi$, we have  
 \begin{equation}\label{Z} 
Z_{\beta}-\mpair{\beta,\ck\alpha}Z_{\alpha}\subseteq 
 Z_{s_{\alpha}(\beta)}.
 \end{equation}
\end{lemma}
 \begin{proof} 
 This  follows directly from \eqref{eq4} and the definition of a root subsystem.
 \end{proof}
 
 \subsection{} The root subsystems of $\Phi^a$ may be  described in terms 
 of solutions of \eqref{Z} as follows.
 \begin{corollary} \label{cor:Zsyst} 
 The root subsystems  of  $\Phi^a$ are in bijective correspondence with pairs $(\Psi, \mset{Z_{\alpha}}_{\alpha\in \Psi})$ where $\Psi$ is a root subsystem 
 of $\Phi$ and  $\set{Z_{\alpha}}_{\alpha\in \Psi}$ is a family of non-empty subsets 
 of $\Int$ satisfying  \eqref{Z} for all $\alpha,\beta\in \Psi$. The correspondence attaches to  $(\Psi, \mset{Z_{\alpha}}_{\alpha\in \Psi})$  the root subsystem  $\Phi^a(\Psi, \mset{Z_{\alpha}}_{\alpha\in \Psi}):=
 \mset{\alpha+n\delta\mid \alpha\in \Psi, n\in Z_{\alpha}}$ of  $\Phi^a$. 
 \end{corollary}
 \begin{proof} Observe that for any solution 
 $\set{Z_{\alpha}}_{\alpha\in \Phi}$ of \eqref{Z},
 $\mset{\alpha\in \Phi\mid Z_{\alpha}\neq \emptyset}$ is clearly a 
 root subsystem $\Psi$ of $\Phi$; hence the corresponding root subsystem
 of $\Phi^a$ is of the required form. Conversely, if $\Psi$ is a subsystem of 
 $\Phi$ and $\mset{Z_{\alpha}}_{\alpha\in \Psi}$ satisfy 
 \eqref{Z} for all $\alpha,\beta\in \Psi$, a solution of \eqref{Z} is obtained by taking
 $Z_{\alpha}=\emptyset$ for $\alpha\in \Phi\setminus \Psi$. The corresponding 
 root subsystem of $\Phi^a$ is then $\Phi^a(\Psi, \mset{Z_{\alpha}}_{\alpha\in \Psi})$.
 \end{proof}

Note that under the above bijection, the root subsystems of $\Phi$ correspond to 
the case where each non-empty subset $\Z_\alpha=\set{0}$.
 
\subsection{}  The above description of root subsystems is convenient for studying
the inclusion relations in the poset of root subsystems of $\Phi^a$. In fact, 
it is obvious that 
\bee \label{eq:rootiinclusion} 
\Phi^a(\Psi, \mset{Z_{\alpha}}_{\alpha\in \Psi})\subseteq 
\Phi^a(\Psi', \mset{Z'_{\alpha}}_{\alpha\in \Psi})
\iff\Psi\subseteq \Psi'\text{ and $Z_{\alpha}\subseteq Z'_{\alpha}$ for all $\alpha\in \Psi'$.} 
 \eee

The next lemma describes the natural action of the extended affine Weyl group on the
root subsystems of $\Phi^a$. This corresponds to the conjugation action 
of $\wt W^a$ on the set of reflection subgroups of its normal subgroup $W^a $.
The action is described for reflections in $W$, and translations, which together
generate $\wt W^a$. 
\begin{lemma}\label{lem:affZact}
 Let  $w\in \wt W^{a}$ and write 
 $w\bigl( \Phi^a(\Psi, \mset{Z_{\alpha}}_{\alpha\in \Psi})\bigr) =
 \Phi(\Psi',\mset{Z'_{\alpha}}_{\alpha\in \Psi'})$. 
\begin{num}\item    If $w\in W$, then $\Psi'=w(\Psi)$ and 
$Z'_{\alpha}=Z_{w^{-1}(\alpha)}$ for $\alpha\in \Psi'$.
\item If $w=t_{\gamma}$ where $\gamma\in P(\Phi)$ (the coweight lattice), then 
$\Psi'=\Psi$ and $Z'_{\alpha}=Z_{\alpha}+\mpair{\alpha,\gamma}$ for all $\alpha\in \Psi'$.  
\end{num}    \end{lemma}
\begin{proof} This follows directly from the definitions and \eqref{eq5}.\end{proof}

\subsection{}  
 For any indecomposable subset $\Gamma$ of $\Phi$, define another 
 subset $\Gamma^{\circ}$ of $V$ as follows: 
 \[\Gamma^{\circ}=\mset{\alpha\mid \alpha\in \Gamma_{\text{\rm short}}}\cup
\mset{\frac{1}{k_{\Gamma}} \alpha\mid \alpha\in \Gamma_{\text{\rm long}}} 
\] 
(see \ref{ss:rootlengthnotation} for the notation).
If $\Psi$ is an indecomposable root system in $V$ with simple system $\Gamma$,  
then $\Psi^\circ$ is a root system in $V$ (homothetic  to $\ck \Psi$, 
and hence isomorphic to  $\ck\Psi$) with simple system  
$\Gamma^{\circ}$, and $W_{\Psi}=W_{\Psi^{\circ}}$.
 
 \subsection{}   Fix an additive subgroup  $Z$ of $\real$. 
 Let $\Psi$ be a root subsystem of $\Phi$, with components
 $\Psi_{1},\ldots, \Psi_n$.   
 
 A subgroup $X$ of  $P_{Z}(\Psi)$ is said to be a $Z$-admissible 
 coweight lattice  for  $\Psi$  if it is of the form  
 $X=\oplus_{i=1}^nX_{i}\subseteq P({\Psi})$
where  either $X_{i}=0$, in which case we set 
$m_{i}=0$, or  $X_{i}=m_{i}P({\Psi_{i}})$ or
$X_{i}=m_{i}P({\Psi_{i}^{\circ}})$ for some  
$m_{i}\in Z\cap \real_{>0}$.

We shall require some facts about the elements $n_{\alpha}\in \Nat$ defined by
$\mpair{\alpha,X}=n_{\alpha}\Int$ for all $\alpha\in \Psi$.
Choose a simple system $\Delta_{i}$  for  $\Psi_{i}$ and set 
$\Delta:=\cup_{i}\Delta_{i}$.
Let $\omega_{\alpha}=\omega_{\alpha}(\Delta)$ for $\alpha\in \Delta$ 
denote the corresponding fundamental coweights. We have 
$X=\oplus_{i=1}^n\oplus_{\alpha\in \Delta_{i}}n'_{\alpha}\Int
\omega_{\alpha}$  where for $\alpha\in \Delta_{i}$,  we set
$n'_{\alpha}=m_{i}k_{\Psi_{i}}$ if  $X_{i}=m_{i}P({\Psi_{i}}^{\circ})$ 
and $\alpha$ is long in $\Psi_{i}$, and otherwise set  $n'_{\alpha}=m_{i}$.  
It follows that if $\alpha\in \Delta$, we have
$n_{\alpha}=n_{\alpha}'$. Since $n_{\alpha}=n_{w(\alpha)}$ for $w\in W_{\Psi}$, 
by the $W_{\Psi}$-invariance of $P(\Psi_{i})$ and $P(\Psi_{i}^{\circ})$, it follows that

\be\label{eq:admprod}  n_{\alpha}=m_{i}k_{i,\alpha}\ee for all $\alpha\in \Psi_{i}$
where $k_{i,\alpha}=k_{\Psi_{i}}$ if $\alpha$ is long in $\Psi_{i}$ and 
$X_{i}=m_{i}P(\Psi_{i}^{\circ})$, and $k_{i,\alpha}=1$ otherwise.

It follows also that
\be \label{eq:Xchar} 
X=\mset{p\in P({\Psi})\mid \mpair{p,\alpha}\subseteq n_{\alpha}\Int
\text{ \rm  for all $\alpha\in \Psi$}}.
\ee
For the left hand side is clearly contained in the right. Conversely,
if an element $x=\sum_{\alpha\in \Delta}a_{\alpha}\omega_{\alpha}$ of $P(\Psi)$ is 
in the right hand side, then for $\alpha\in \Delta_{i}$ we have 
$a_{\alpha}=\mpair{x,\alpha}\in n_{\alpha}\Int=n_{\alpha'}\Int$ and so  $x\in X$.
We shall also require the following divisibility property, which
is a consequence of \eqref{eq:admprod}.  
\be\label{eq19}\mpair{\beta,\ck\alpha}n_{\alpha}\Int \subseteq n_{\beta}\Int
\ee 
for $\alpha,\beta\in \Psi$.  This is trivial except in the 
situation where $\alpha,\beta$ are in the same 
component $\Psi_{i}$, are not orthogonal, $X_{i}=m_{i}P({\Psi_{i}}^{\circ})$ and 
$\beta$ is long and $\alpha$ is short in $\Psi_{i}$; but then 
$\mpair{\beta,\ck\alpha}=k_{\Psi}$ by rank two considerations and 
the result remains valid.

The $Z$-admissible coroot lattice (for $\Psi$) corresponding to $X$ 
is defined to be the group 
$Y=\sum_{i}Y_{i}$ where $Y_{i}=0$ if $X_{i}=0$,  and $Y_{i}=m_{i}Q({\Psi_{i}})$
(resp., $Y_{i}=m_{i}Q(\Psi^{\circ}_{i})$ if $X_{i}=P({\Psi_{i}})$ 
(resp., $X_{i}=P({\Psi^{\circ}_{i}}))$.
This gives a natural bijective correspondence between $Z$-admissible 
coweight lattices and $Z$-admissible coroot lattices of $\Psi$.

\subsection{Second parameterisation} The next theorem gives our second 
main parameterisation of the root subsystems
of $\Phi^a$. To formulate it, let $\RR'(\Phi,Z)$ denote the set of all pairs $(\Psi,X)$  
where $\Psi$ is a root subsystem of $\Phi$ and 
$X\in P_{Z}(\Psi)/X'$ is a coset $X=a+X'$, where $a\in P_{Z}({\Psi})$,  
of some $Z$-admissible coweight lattice $X'$ of ${\Psi}$.
\begin{theorem}\label{thm:rootcoset}
 There is a natural bijection between $\RR'(\Phi,\Int)$ and the set 
 of root subsystems of $\Phi^a$. 
The bijection attaches to $(\Psi,X)\in \RR'(\Phi,\Z)$ the root subsystem
$\Phi^a(\Psi,X):=\Phi^a(\Psi,\set{Z_{\alpha}}_{\alpha\in \Psi})$ where 
$Z_{\alpha} :=\mset{\mpair{\alpha,x}\mid x\in X}\subseteq \Int$ and    
$ \Phi^{a}(\Psi,\set{Z_{\alpha}}_{\alpha\in \Psi}):=
\mset{\alpha+n\delta\mid \alpha\in \Psi, n\in Z_{\alpha}}$.
\end{theorem}    
\begin{proof} First we show that there is a map 
$g\colon (\Psi,X)\mapsto \Phi^a(\Psi,X)$ as indicated   from $\RR'(\Phi,\Int)$ to the set of 
root subsystems of $\Phi^{a}$. Fix $(\Psi,X)\in \RR'(\Phi,\Int)$ 
and define $Z_{\alpha}:=\mpair{\alpha,X}\subseteq \Int$. for $\alpha\in \Psi$. 
Write $X=a+X'$ where $X'$ is an admissible coweight lattice for $\Psi$ 
and $a\in P({\Psi})$. Note that  $Z_{\alpha}=\mpair{a,\alpha}+\Int n_{\alpha}$ 
for the unique integer $n_{\alpha}\geq 0$, such that 
$n_{\alpha}\Int=\mpair{X',\alpha}$. 
It is easy to check from \eqref{eq:admprod} and  
\eqref{eq19} that the pair $(\Psi,\mset{Z_{\alpha}}_{\alpha\in \Psi})$ 
satisfies the condition of Corollary \ref{cor:Zsyst},
and so determines a root subsystem 
$g(\Psi,X)=\Phi^{a}(\Psi,\mset{Z_{\alpha}}_{\alpha\in \Psi})$ as
in Corollary \ref{cor:Zsyst}.

 We now  show  that $g$ is injective.  
Fix $(\Psi,X)\in \RR'(\Phi,\Int)$ 
and define $Z_{\alpha}:=\mpair{\alpha,X}\subseteq \Int$ for $\alpha\in \Psi$ as above. 
Since 
\bee \Psi=\mset{\alpha\in \Phi\mid \alpha+n\delta\in \Phi^{a}(\Psi,X)
\text{ \rm for some $n\in \Int$,}}
\eee 
it will suffice to show that 
$\mset{Z_{\alpha}}_{\alpha\in \Psi}$ uniquely determines $(\Psi,X)$.
It  follows from \eqref{eq:Xchar} that
 \bee \label{eq:Xprimechar} 
X'=\mset{p\in P({\Psi})\mid \mpair{p,\alpha}\subseteq n_{\alpha}\Int
\text{ \rm  for all $\alpha\in \Psi$.}}
\eee
From this, we obtain
$ X=\mset{p\in P({\Psi})\mid \mpair{p,\alpha}
\subseteq Z_{\alpha}\text{ \rm  for all $\alpha\in \Psi$}}$. 
The injectivity of the map $g$ follows.

Now to show  that  $g$
is surjective, it will suffice to show that for $(\Gamma,f)$ as in Theorem \ref{thm1}, 
there is some pair $(\Psi,X)$ with
$\Phi^{a}(\Gamma,f)=\Phi^{a}(\Psi,X)$. In the following argument, we use the notation 
of \ref{ss:nipnot} concerning $\Gamma$. 
Choose $\Psi=\Sigma$ and  $p\in P({\Sigma})$ such that $\mpair{p,\alpha}=r_{\alpha}$ for all 
$\alpha\in \Gamma'$ and  define $X'=\oplus_{i}X_{i}\subseteq P({\Sigma})$
where
\be \label{eq:compare} 
X'_{i}:=\begin{cases}0,&\text{if $\Gamma'=\Gamma_{i}$}\\
K_{i}P({\Gamma_{i}}),&\text{if  $\Gamma'\neq \Gamma_{i}$ and $\alpha_{i,0}$ is long}\\
K_{i}P({\Gamma_{i}}^{\circ}),&\text{if  $\Gamma'\neq \Gamma_{i}$ and $\alpha_{i,0}$ is short.}
\end{cases}
\ee 
Then  $X'=\oplus X'_{i}$ is a $\Int$-admissible coweight lattice  of $P({\Sigma})$ and 
from  Proposition \ref{prop:niprootsyst} and \eqref{eq:admprod}, we see that  $\Phi^a(\Sigma, p+X')=\Phi^{a}(\Gamma,f)$.
This completes the proof.
\end{proof}

\subsection{} \label{ss:bij1}
Evidently there is a unique bijection  $j=j_{\Int}\colon \RR(\Phi,\Int)\rightarrow \RR'(\Phi,\Int)$ such that, 
setting $j(\Gamma,f)=(\Psi,X)$, we have $\Phi^{a}(\Psi,X)=\Phi^{a}(\Gamma,f)$.
The above proof indicates how to compute $j(\Gamma,f)$ explicitly from $(\Gamma,f)$.
Computing $j^{-1}(\Psi,X)$ amounts to determining the canonical simple system  
of $\Phi^{a}(\Psi,X)$. 
We describe below one simple  way in which this may be done;
a more natural  procedure  using alcove geometry will be described in \S \ref{comparison}. 

Suppose $(\Psi,X)\in \RR'(\Phi,\Int)$. We may find
$(\Gamma,f)\in \RR(\Phi,\Int)$  with $\Phi^a(\Gamma,f)=\Phi^{a}(\Psi,X)$ as follows.
For $\alpha\in \Psi$, write $Z_{\alpha}=\mpair{X,\alpha}\subseteq \Int$. 
and $Z'_{\alpha}:=\mset{n\in Z_{\alpha}\mid\alpha+n\delta\in \Phi^a_{+}}$. 
Let $\Psi':=\mset{\alpha\in \Psi\mid Z'_{\alpha}\neq \emptyset}$ and for
$\alpha\in \Psi'$, let $r'_{\alpha}:=\min(Z'_{\alpha})$.
Let $\Delta$ denote the simple system contained in $\Phi^a_{+}$ of $\Phi^a(\Psi,X)$ and  
$\Delta'':=\mset{\alpha+r'_{\alpha}\alpha\mid \alpha\in \Psi'}$.
One then checks that $\Delta\subseteq \Delta''$.
It is well known that $\Delta$ is the  unique inclusion-minimal subset $\Delta'$ of 
$\Phi^{a}_{+}\cap \Phi^a(\Psi,X)$ such that $\Phi^{a}(\Psi,X)\cap \Phi^a_{+}\subseteq \Nat\Delta$.
It follows that $\Delta$ is the unique inclusion-minimal subset $\Delta'$ of 
the finite set  $\Delta''$  such that $\Delta''\subseteq \Nat\Delta'$. Thus,  
$\Delta$ in $\Phi^{a}_{+}$ may be effectively determined. 
Set $\Gamma=\pi( \Delta)$ (a \nip in $\Phi$) and define 
$f\colon \Gamma\rightarrow \Int_{\geq 0}$ by $f(\alpha)=r'_{\alpha}$. Then $\Phi^a(\Gamma,f)=\Phi^{a}(\Psi,X)$.

\subsection{} 
We may now use the action of $\wt  W^{a}$ on $V$  to describe the $\wt W^{a}$-action 
on root subsystems of $\Phi^{a}$ in terms of their parameterisation in 
Theorem \ref{thm:rootcoset}. For any subset $\Sigma$ of $\Phi$, we let
$p_{\Sigma}\colon V\rightarrow \real \Sigma$ denote the orthogonal projection
of $V$ on $\real\Sigma$.
\begin{corollary}\label{cor:affcosetact}
Let  $w=t_{\gamma}x\in \wt W^{a}$ where $x\in W$ and $\gamma\in P(\Phi)$.
Let $(\Psi,X)\in \RR'(\Phi,\Int).$  Then $\Psi':=x(\Psi)$ is a root 
subsystem of $\Phi$ and $X'':=p(\tau_{\gamma}x(X))$ is a coset in 
$P({\Psi'})$ of some admissible coweight lattice of $P({\Psi'})$, where 
$p=p_{\Psi'}$.\ Further, $w\Phi^a(\Psi,X)=\Phi^a(\Psi',X'')$.\end{corollary}

\begin{remark} Note that  the isomorphism  
$t_{\gamma}x\leftrightarrow \tau_{\gamma}x$
between the affine and linear versions of $\wt W^{a}$ used 
implicitly above is different from the one
($t_{\gamma}x\leftrightarrow \tau_{-\gamma}x$) used in
Proposition \ref{prop:afflinear}.
\end{remark}
\begin{proof}  Write $X=a+X'$ where $a\in 
P({\Psi})$ and  $X'$ is an admissible coweight lattice  of $P({\Psi})$.
Clearly, $\Psi'$ is a root subsystem of $\Phi$ and $X''':=x(X')$ is
an admissible coweight lattice  of
$P({\Psi'})$. We have $\tau_{\gamma}x(X)=\gamma+x(a)+X'''$ 
where clearly $x(a)\in P({\Psi'}) $. Hence $p\tau_{\gamma}x(X)=p(\gamma)+x(a)+X''$. 
To prove the assertion about $X''$, it will suffice to show that 
$p(\gamma)\in P({\Psi'})$. But this holds since 
for any $\alpha\in \Psi'$, we have $\mpair{\alpha,p(\gamma)}=\mpair{\alpha,\gamma}\in \Int$. 
This proves the first claim. The final claim follows from 
Lemma   \ref{lem:affZact} as follows. Write $Z_{\alpha}=\mpair{X,\alpha}$ for $\alpha\in \Psi$. 
Then by Lemma \ref{lem:affZact}, 
$ w\Phi^{a}(\Psi,X)=w\Phi^{a}(\Psi,\set{Z_{\alpha}}_{\alpha\in \Psi})=
\Phi^a(\Psi',\set{Z'_{\alpha}}_{\alpha\in \Psi'})$ where 
for $\alpha\in \Psi'$, $Z'_{\alpha}=Z_{x^{-1}\alpha}+\mpair{\alpha,\gamma}=\mpair{x^{-1}\alpha,X}+\mpair{\alpha,\gamma}$.
But from the above,  $X''=p\tau_{\gamma}x(X)=p(\gamma)+x(a)+x(X')=p(\gamma)+x(X)$, and so 
\[ 
\mpair{\alpha,X''}=\mpair{\alpha,p(\gamma)+x(X)}=
\mpair{\alpha,p(\gamma)}+\mpair{\alpha, x(X)}=\mpair{\alpha,\gamma}+\mpair {x^{-1}\alpha,X}.
\]  
Hence $\Phi^a(\Psi',\set{Z'_{\alpha}}_{\alpha\in \Psi'})=\Phi^a(\Psi',X')$ as required.
\end{proof}

\subsection{Affine versus finite} 
Now fix $(\Psi,X)\in \RR'(\Phi,\Int)$ and write $X=b+X'$, where $b\in P({\Psi})$ and
$X'$ is an admissible coweight lattice  of $P({\Psi})$.
Let $\set{\Psi_{i}}$ be the components of $\Psi$, 
and write $X'=\sum_{i}X'_{i}$ where for each $i$, 
either $X'_{i}=m_{i}P({\Psi_{i}})$, $X'_{i}=m_{i}P({\Psi_{i}^{\circ})}$ 
with $m_{i}\in \Int_{>0}$
or $X'_{i}=0$ and $m_{i}:=0$. 
Set $Y_{i}:=m_{i}Q(\Psi_{i})$, $Y_{i}:=m_{i}Q(\Psi_{i}^\circ)$ or $Y_{i}:=0$ 
accordingly, and define 
$Y=\sum_{i}Y_{i}$. 
It is clear that $\Phi^{a}(\Psi,X)$ is finite  
(equivalently, $W^{a}(\Psi,X)$  is finite) if and 
only if $X'=\set{0}$. Also, $W^{a}(\Psi,X)$ has only 
affine type  components (i.e.  no finite components) if 
and only if   $X'_{i}\neq \set{0}$ for all $i$.
Finally, $W^{a}(\Psi,X)$ is of finite index in $W^{a}$ if 
and only if its translation subgroup is of finite index in that 
of $W^{a}$, which is the case if and only if $X'_{i}\neq 0$ for 
all $i$ and $\real \Phi=\real \Psi$. 
 
\subsection{} We now  describe a decomposition of $W^{a}(\Psi,X)$ as the 
semidirect product of its translation subgroup and a finite reflection group, 
and explicitly describe its elements. 
 \begin{proposition}\label{prop:subgpelts}
 Let $X=b+X'$ be a coset in $P({\Psi})$ of an admissible coweight 
 lattice $X'$ for  $\Psi$, and let $Y$ be the admissible coroot lattice 
 corresponding to $X'$. Then\begin{num}\item $W^{a}(\Psi, X)$ is a semidirect product
 $W^{a}(\Psi, X)=W''\ltimes t_{Y}$ of the finite reflection subgroup 
 $W'':=\mset{t_{b-w(b)}w\mid w\in W_{\Psi}}$ by the translation   group $t_{Y}:=\mset{t_{\gamma}\mid \gamma\in Y}$.
 \item $W^{a}(\Psi,X)=\mset{t_{b-w(b)+\gamma}w\mid w\in W_{\Psi},\gamma\in Y}$.
 \end{num}\end{proposition}
 \begin{proof} A typical root of $\Phi^a(\Psi,X)$ is 
 of the form $\alpha+\mpair{\alpha,b+z}\delta$ 
 where $\alpha\in \Psi$ and  $z\in X'$, with corresponding reflection 
 $s_{\alpha+\mpair{\alpha,b+z}\delta} =t_{\mpair{\alpha,b+z}\ck \alpha}s_{\alpha}$.
 For each $\alpha\in \Psi$,  we have $s'_{\alpha}:=s_{\alpha+\mpair{\alpha,b}\delta}=
 t_{\mpair{\alpha,b}\ck\alpha}s_{\alpha}=t_{b-s_{\alpha}b}s_{\alpha}\in W^{a}(\Psi,X)$. 
   It is easy to check that the  reflection  subgroup of 
   $W^{a}$ generated by $s'_{\alpha}$ for $\alpha\in \Psi$ is $W''$. 
   For any affine Weyl group, the group of translations is generated by 
   the translations which are products of two reflections.  Therefore, the 
   subgroup $T'$ of translations of $W^a (\Psi,X)$ is generated by elements  
  $s_{\alpha+\mpair{\alpha,b+z}\delta}s_{\alpha+\mpair{\alpha,b}\delta}=
  t_{\mpair{\alpha,z}\ck \alpha}$ for $z\in X'$ and $\alpha\in \Psi$.  
  Write $\mpair{\alpha,X'}=n_{\alpha}\Int$, so $T'$ is generated by
  $t_{n_{\alpha}\ck \alpha}$ for $\alpha\in \Psi$. But it  follows easily from \eqref{eq:admprod}
  that $T'=t_{Y}$.    The rest of (a) follows readily, and (b) follows immediately from  (a). 
   \end{proof}

\subsection{} The next result describes the pointwise stabiliser  
in $\wt W$ of a root subsystem and
the centraliser in $\wt W$ of the corresponding reflection subgroup, where $\wt W$ is as in 
\ref{genextaff}.

\begin{proposition}\label{prop:fixer}
\begin{num}\item The pointwise stabiliser of $\Phi^a(\Psi,X)$ in $\wt W$
is the semidirect product $W''\rtimes T$ where $W''$ is the parabolic subgroup 
$W''=W_{\Phi\cap \Psi^{\perp}} $ of $W$ and $T$ is the group of translations
$T=t_{R\cap \Psi^{\perp}}:=\mset{t_{\gamma}\mid \gamma\in R\cap \Psi^{\perp}}$.
\item For $w\in W$ and $\gamma\in R$, $wt_{\gamma}$ centralizes $W^{a}(\Psi,X)$ 
if and only if for each component $\Psi_{i}$ of $\Psi$, one of the two conditions (i) or (ii) below hold:

(i) for all $\alpha\in \Psi_{i}$, $w(\alpha)=\alpha$ and $\mpair{\gamma,\alpha}=0$

(ii) $X_{i}=0$
 and for all $\alpha\in \Psi_{i}$, $w(\alpha)=-\alpha$ and $\mpair{\gamma,\alpha}=
 -2\mpair{a,\alpha}$ where $a$ is a representative of  the coset $X$. 
\end{num} 
\end{proposition}
\begin{proof} Part (a) follows  from   \eqref{eq5}, using the fact 
that the centraliser in $W$ of any subset $L$ of $V$ is the parabolic 
subgroup of $W$  generated by $s_{\alpha}$ for $\alpha\in \Phi\cap L^{\perp}$. 
For part (b),  one notes first  
that   $x\in \wt W$  centralises  
$W^{a}(\Psi,X)$ if and only if  for each $\beta\in \Phi^{a}(\Psi,X)$, 
we have $x\beta=\pm \beta$. It is readily verified that this suffices to conclude that 
$wt_{\gamma}$ is in the centraliser. 

To prove necessity in (b),  one notes first  
that if  $x\in \wt W$  centralises  
$W^{a}(\Psi,X)$, then  for each $\beta\in \Phi^{a}(\Psi,X)$, 
we have $x\beta=\pm \beta$.
Writing $x=wt_{\gamma}$ and $\beta=\alpha+n\delta$, \eqref{eq5} implies that 
$w\alpha=\pm \alpha$ for each $\alpha\in \Psi$.  By irreducibility, for each component 
$\Psi_{i}$ of $\Psi$, we then have either $w\alpha=\alpha$ for all $\alpha\in \Psi_{i}$ or 
$w\alpha=-\alpha$ for all $\alpha\in \Psi_{i}$. A further application of \eqref{eq5}
now shows that either (i) or (ii) holds.
\end{proof}

\subsection{} We describe the $\wt W$-orbit  
of $\Phi^a(\Psi,X)$ and the  setwise stabiliser of $\Phi^a(\Psi,X)$ 
in $\wt W$ (which  coincides with  the normaliser in $\wt W$ of $W^{a}(\Psi,X)$).
For this, we consider a second root system $\Phi^a(\Psi_{1},X_{1})$,
for which we use similar notation as for $\Phi^a(\Psi,X)$. In particular, 
write $X=a+X'$ where $a$ is a representative of $X$ and  $X_{1}=a_{1}+X_{1}'$,  etc). 
\begin{proposition}\label{prop:normaliser}
Let  $w\in W$ and $\gamma\in R$. Then 
\begin{num}
\item    $t_{\gamma}w(\Phi^a(\Psi_{1},X_{1})=\Phi^a(\Psi,X)$ 
if and only if $w(\Psi_{1})=\Psi$, $w(X_{1}')=X'$ and 
$a\in w(a_{1})+p_{\Psi}(\gamma)+X'$.
\item    $\Phi^a(\Psi_{1},X_{1}) $ and 
$ \Phi^a(\Psi,X)$ are in the same $\wt W$-orbit
if there exists $w'\in W$ with $w'(\Psi_{1})=\Psi$, $w'(X_{1}')=X'$ and 
$a-w'(a_{1})\in p_{\Psi}(R')+X'$.
\item $t_{\gamma}w$ normalises $W^{a}(\Psi,X)$ 
if and only if  $w(\Psi)=\Psi$, $w(X')=X'$ and 
$p_{\Psi}(\gamma)\in a-w(a)+X'$.
\end{num} \end{proposition}
\begin{proof} Part (a) follows from  
Corollary \ref{cor:affcosetact}, while (b) and (c) follow readily from (a).
We omit the details.
\end{proof}

\section{Comparison of the parameterisations}  \label{comparison}
\subsection{}  Consider  the bijection (discussed in \ref{ss:bij1})  
$j=j_{\Int}\colon \RR(\Phi,\Int)\rightarrow \RR'(\Phi,\Int)$, 
where both sides are parameter sets for the set of reflection subgroups of $W^{a}$. 
We indicate below  how $j$ arises  as the restriction (to ``integer points'') of 
a bijection  $j_{\real}\colon   \RR(\Phi,\real)\rightarrow \RR'(\Phi,\real)$  
where  both sides correspond naturally to the  discrete reflection subgroups 
of the semidirect product  $W\ltimes \tau_{V}$.

  Let $\rh\Phi:=\Phi+\real \delta$  and 
  \[\rh \Phi_{+}:=\mset{\alpha+n\delta\mid \alpha\in \Phi,n\in \real_{\geq 0}, n>0 
  \text{ \rm if $\alpha\in -\Phi_{+}$}}.\] The group $\rh W^{a}:=\mpair{s_{\alpha}\mid \alpha\in \rh \Phi}$
  is a (generally non-discrete) group of linear transformations of $\wh V$.
  Note that $\rh W^{a}$ is naturally identified 
  as the group of affine transformations  of $V$ generated by $W$
  and  the group of translations $\tau_{V}=\mset{\tau_{v}\mid v\in V}$ of $V$;
  under this identification, $\rh W^{a}$ is a semidirect product
  $\rh W^{a}=W\ltimes \tau_{V}$.

 We define root subsystems of $\rh \Phi^{a}$ as for $\Phi^{a}$.
 A root subsystem $\Psi$  of $\rh \Phi^{a}$ is discrete (as a subset of $\wh V$ in
 the usual topology) if and only if for each $\alpha\in \Phi$,
 the set $Z_{\alpha}:=\mset{n\in \real\mid \alpha+n\delta\in\Psi}$ is a discrete subset of $\real$.
 It is easy to see that this holds precisely when the corresponding reflection subgroup
 $\mpair{s_{\alpha}\mid \alpha\in \Psi}$ is a discrete subgroup of $\text{\rm GL}(\wh V)$. 
 We define simple subsystems of $\rh \Phi^{a}$ as positively independent
 \nips in $\rh \Phi_{+}$.
  There are natural bijective correspondences between discrete root 
  subsystems of $\rh \Phi^{a}$,
   discrete reflection subgroups of $\rh W$, and simple
   subsystems of $\rh\Phi$ which are contained in $\rh\Phi_{+}$. 
   The simple subsystem of $\rh \Phi$ corresponding to a discrete 
   reflection subgroup $W'$ of $\rh W$ and contained in $\rh \Phi_{+}$ will be
written $\Delta_{W'}$.   
   
   \subsection{} 
   The main results (Theorems \ref{thm1} and \ref{thm:rootcoset}) hold in this
   more general context, with essentially the same proofs, but modified as
   follows: one replaces 
      ``reflection subgroups of the affine Weyl group $W^{a}$'' by
    ``discrete reflection subgroups of $\rh W^{a}$'',   ``root subsystems of $\Phi^{a}$''  by
    ``discrete root  subsystems of $\rh \Phi^{a}$,''  ``$Z=\Int$'' by ``$Z=\real$''
    (and correspondingly ``$\RR(\Phi,\Int)$'' by ``$\RR(\Phi,\real)$'' etc) and 
  ``$\Phi^{a}$''  by $\rh \Phi$ (and correspondingly  ``$\Phi^{a}(\Gamma,f)$''  by
  $\rh\Phi^{a}(\Gamma,f)$'' etc).
  
  This gives a natural  bijection $j_{\real}\colon \RR(\Phi,\real)\rightarrow \RR(\Phi,\real)$ 
  which extends $j=j_{\Int}$, since both sides are naturally in bijection with 
  discrete root subsystems of $\rh \Phi^{a}$. The bijection $j_{\real}$ has 
  similar properties to those of $j_{\Int}$ as  discussed in \ref{ss:bij1}.
  
  \begin{remark}
  In fact, most of the  theory (with obvious exceptions such as Corollary \ref{cor:alcvol}(b)) 
  established in the previous Sections \S \ref{s1}--\ref{s5}
  for reflection subgroups of $W^{a}$ 
   extends mutatis mutandis to discrete reflection subgroups of $\rh W^{a}$.
  \end{remark}

\subsection{}  Let $k\subseteq \real$ denote either $\Int$ or $\real$. For 
a root subsystem $\Psi$  of $\Phi$  and 
a $k$-admissible subgroup  $X'=\sum_{i}X'_{i}$  of $P_{k}({\Psi})$, we define 
  \bee 
  \RR'_{(\Psi,X')}(\Phi,k):=\mset{(\Psi,X)\in \RR'(\Phi,k)\mid 
X\in P_{k}({\Psi})/X'}.
\eee    
and  
 $ \RR_{(\Psi,X')}(\Phi,k):=j_{k}^{-1}\bigl( \RR'_{(\Psi,X')}(\Phi,k)\bigr)$. 
 Clearly, $j_{k}$ restricts to a bijection 
 \[j_{k,\Psi,X'}\colon \RR_{(\Psi,X')}(\Phi,k)\rightarrow  \RR'_{(\Psi,X')}(\Phi,k).
 \]
  We have $\RR'(\Phi,k)=\dot \cup_{(\Psi,X')} \RR'_{(\Psi,X')}(\Phi,k) $ 
(and similarly with $\RR'$ replaced by $\RR$) where the union is over 
pairs $(\Psi,X')$ such that $\Psi$ is a root subsystem of $\Phi$ and $X'$ is
a $ k$-admissible coweight lattice in $P_{\real}(\Psi)$. 

A preliminary discussion of the bijections $j_{\Int}^{\pm 1}$ is given in \S\ref{ss:bij1}. 
The description of $j_{\Int}$ there is reasonably natural and  explicit. 
Here we wish to  provide a more natural interpretation of $j_{\Int}^{-1}$ than 
was available previously. Now  $j_{\Int}^{-1}$ is obtained by restriction 
of $j_{\real}^{-1}$, so it sufficient to describe the maps $j_{\real,\Psi,X'}^{-1}$. 
The main result of this section is Theorem \ref{thm:bij}, 
which express $j_{\real,\Psi,X'}^{-1}$ as a composite of maps which 
have natural interpretations in terms of  the alcove geometry of a 
certain affine Weyl group $W'=W'_{\Psi,X'}$ attached to $(\Psi,X')$.

 \subsection{} 
 Fix a root subsystem $\Psi$  of $\Phi$  and 
 an $\real$-admissible subgroup  $X'=\sum_{i}X'_{i}$  of $P_{\real}({\Psi})$.
    Let $Y'$ be the $\real$-admissible coroot lattice of 
 $\Psi$ corresponding to $X'$, and let $W'=W'_{\Psi,X}$ be the group of affine transformations
 of $V$ generated by  $W_{\Psi}\cup \tau_{Y'}$. Note that $W'$ acts 
 faithfully by restriction as a discrete affine  reflection group  on $\real\Psi$; 
 we shall think of $W'$ in this way. It is isomorphic as Coxeter group
 to $W^a (\Psi, X)$ for any $ X\in P_{\real}({\Psi})/X'$. 
 \begin{remark} For fixed $\Psi$ and varying $X$, the subgroups $W'_{\Psi, X}$
are precisely the discrete reflection groups acting in $\real\Psi$ 
whose linear parts are (the restriction to $\real \Psi$ of)
 $W_{\Psi}$ and which have $0$ as a special point.\end{remark}
 

 \subsection{Cosets and a fundamental domain}\label{ss:defh} 
 Denote by  $\Gamma'_{0}$ the simple system of $\Psi$ with 
 $\Gamma'_{0}\subseteq \Phi_{+}$  and define  the corresponding coweights 
 $\omega_{\alpha}':=\omega_{\alpha}(\Gamma'_{0})$ for $\alpha\in \Gamma'_{0}$; thus
 $\mpair{\alpha,\omega'_{\beta}}=\delta_{\alpha,\beta}$ if $\alpha,\beta\in \Gamma'_{0}$.  
  The group  $X'$ acts by translation on $P_{\real}( \Psi)=\real \Psi$.  
  Since  $ X'=\sum_{i:X'_{i}\neq 0}\sum_{\alpha\in \Psi_{i}}\Int m_{i}k_{i,\alpha}\omega_{\alpha}$,  
  it follows that the region 
 \bee D:=\mset{\sum_{\alpha\in \Gamma'_{0}}e_{\alpha}\omega'_{\alpha}\mid e_{\alpha}\in \real,
 \text{  
  $0\leq e_{\alpha}< K_{i}k_{i,\alpha}$ if  $X'_{i}\neq 0$ and $\alpha\in \Psi_{i}$}}
  \subseteq \real \Psi \eee
  is   a fundamental domain for the action of $X'$ by translation on $\real \Psi$.
  This gives  a  bijection
  $h\colon P_{\real}(\Psi)/X'\cong D$ such that $h^{-1}(d)=d+X'$ for all $d\in D$.

\subsection{} Let  $\mc{B}$ denote the set of  alcoves of 
$W'$ acting on $\real \Psi$, i.e. the connected components of the 
complements of reflecting hyperplanes of $W'$ in $\real \Psi$.
Any  alcove $B$ is a  direct sum of simplices and simplicial 
cones lying in the pairwise orthogonal  subspaces of 
$\real \Psi$ given by the linear spans of the components $\Psi_{i}$ of $\Psi$. 
Each alcove $B$ may be written as   
\bee B=\mset{v\in \real\Psi\mid  \mpair{\alpha,v}>c_{\alpha}
\text{ for all $\alpha\in \Sigma$}}
\eee   
where each $c_{\alpha}\in \real$,   $\Sigma=\Sigma_{B}\subseteq\Psi$ 
and $\Sigma$ is minimal under inclusion; the hyperplanes 
$\mset{v\in \real \Psi\mid \mpair{\alpha,v}=c_{\alpha}}$ 
for $\alpha\in \Sigma_{B}$ are the walls of $B$. We 
define the lower closure  $B_{0}$ of $B$  as
\begin{multline*}B_{0}=\mset{v\in \real\Psi\mid 
\mpair{\alpha,v}\geq c_{\alpha}\text{ for all $\alpha\in \Sigma\cap -\Phi_{+}$}, \\
\mpair{\alpha,v}> c_{\alpha}\text{ for all $\alpha\in \Sigma\cap \Phi_{+}$}}. 
\end{multline*}
Thus $B_{0}$ is obtained by removing all points  
of  some of the walls of $B$ from the topological closure 
$\overline B=\mset{v\in \real\Psi\mid  
\mpair{\alpha,v}\geq c_{\alpha}\text{ for all $\alpha\in \Sigma$}}$ 
of $B$; the walls removed are precisely  those whose unit normal 
(outward from $B$ in $\real\Psi$) is a positive scalar multiple of 
a positive root (i.e one in $\Psi\cap \Phi_{+}$). 
\subsection{Alcoves} Basic properties of alcoves and their lower 
closures are recorded in the following Lemma. 
\begin{lemma}\label{lem:lowclos}
\begin{num}\item
 $\real\Psi=\dot \cup_{B\in \mc{B}}B_{0}$  (disjoint union).
 \item If $B\in \mc{B}$ and $\lambda\in X'$, then $C:=B+\lambda\in \mc{B}$ and 
 $C_{0}=B_{0}+\lambda$.
 \item Every $v\in X'$ is contained in $B_{0}$ for a unique $B\in \mc{B}$.
 \item Every $v\in Y'$ is contained in $\overline{B}$ for a unique $B\in \mc{B}$.
 \ \item Let $B$ be an alcove for $W'$ on $\real \Psi$ with $0\in \overline{B}$.
 Let  $G:=\mset{C\in \mc{B}\mid 0\in \overline{C}}$. 
 Then  $G=\mset{wB\mid w\in W_{\Psi}}$. 
 \item  The region $D'=\dot\cup_{C\in G}C_{0}$ is a fundamental 
 domain for the action on $\real \Psi$  of the translation subgroup 
 $T=\mset{x\mapsto x+v\colon \real\Psi\rightarrow \real \Psi\mid 
v\in Y'}$  of   $W'$.  \end{num}
  \end{lemma}
  \begin{proof} The notion of lower closure of an alcove is related in an 
  obvious way to the notion of upper closure of a facet as discussed in 
  \cite[6.1--6.2]{Jan}. By simple modification of the  discussion in 
  \cite[6.2]{Jan}, one sees that (a)--(c) hold. For (d)--(e), 
  see \cite[Ch V, \S 3, no. 10]{Bour}.
  Finally, (f) follows from (a),(b), (e)  and Proposition  \ref{prop:affinewequiv}(a) 
  applied to $W'$.
  \end{proof} 
  
\subsection{}\label{ss:defk}
By definition, $D'$ is a union of lower closures of alcoves. We observe 
that the same holds for $D$. 
\begin{lemma}
\label{lem:transdom} We have $D=\bigcup_{\substack{B\in \mc{B}\\ B\subseteq D }}B_{0}$.
\end{lemma}
 \begin{proof} Observe that  
 \bee D:=\mset{v\in \real \Psi\mid 0\leq 
 \mpair{v,\alpha}< K_{i}k_{i,\alpha}\text{\rm  if  $X'_{i}\neq 0$ and 
 $\alpha\in \Psi_{i}\cap \Gamma_{0}$}}\eee
 It is easy to check that the walls of $D$ are reflecting 
 hyperplanes of $W'$.  The desired conclusion follows from 
 Lemma \ref{lem:lowclos}(a) since  the definition of lower closure 
 and the form of the above  inequalities defining $D$  imply that 
 if $v\in D$ is in $B_{0}$ for some  $B\in \mc{B}$, then $\mc{B}_{0}\subseteq D$.
 \end{proof}
 
\subsection{} Since $D'$ is a fundamental domain for $Y'$ acting by translation 
on $\real\Psi$, there is a unique map  $k\colon D\rightarrow D'$ such that 
$k(d)-d\in Y'$ for all $d\in D$. 
For any alcove $B\in \mc{B}$ with $B\subseteq D$, there is 
a unique $\gamma_{B}\in Y'$ such that $B_{0}+\gamma_{B}\subseteq  D'$. 
The map $k$ is determined by the condition 
that $k(d)=d+\gamma_{B}$ for all $ B\in \mc{B}$ with $B\subseteq D$ and
all $d\in B_{0}$. Obviously $k$ is injective, and the image of $k$ is the 
union of the lower closed alcoves in $D$ which are translates by elements 
of $Y'$ of lower closed alcoves in $D$.

\begin{remark} Let $B$ be some alcove of $W'$ with $0\in \overline{B}$.
Consider the subgroup 
$G=\mset{w\in W_{\Psi}\mid wB=\tau_{\gamma}B\text{ \rm for some $\gamma\in X'$}}$. 
It is well known   (and follows from Lemma \ref{lem:lowclos} and 
Proposition \ref{prop:affinewequiv}(a))  
that $G\cong X'/Y'$, under the map taking each 
$w\in G$ to the coset  $\gamma+Y'$ where $\gamma\in X'$ with $wB=\tau_{\gamma}B$. 
We have $\text{\rm Im}(k)\subseteq D'=\cup_{w\in W_{\psi}}(wB)_{0}$. 
Let $M:=\mset{w\in W_{\Psi}\mid (wB)_{0}\subseteq \text{\rm Im}(k)}$.
It is not difficult to see that $M$ is  a set of  coset 
representatives for $W'/G$. \end{remark}

\subsection{}  Recall that  
$ \RR_{(\Psi,X')}(\Phi,\real):=j_{\real}^{-1}\bigl( \RR'_{(\Psi,X')}(\Phi,\real)\bigr)$. 
We shall require the following detailed description of this subset of $\RR(\Phi,\real)$.
 \begin{lemma} \label{lem:bijlem1}
  Let $(\Gamma, f)\in \RR(\Phi,\real)$.
 Then $(\Gamma,f)\in \RR_{(\Psi,X')}(\Phi,\real)$ if and only if the following conditions $\text{\rm (i)--(iv)}$ hold:
\begin{conds}\item there is some simple subsystem $\Gamma'$  of 
 $\Psi$ with  $\Psi\supseteq\Gamma\supseteq \Gamma'$. 
 \item if  $X_{i}'=0 $ then $\Gamma\cap \Psi_{i}= \Gamma'\cap \Psi_{i}$. 
\item if $X_{i}'=m_{i}P({\Psi_{i}})$ where $m_{i}>0$  then 
$(\Gamma\setminus \Gamma')\cap \Psi_{i}=\set{\alpha_{i,0}}$ 
where $-\alpha_{i,0}$ is the  highest root of $\Psi_{i}$ with 
respect to its simple system $\Psi_{i}\cap \Gamma'$, and $K_{i}=m_{i}$.
\item if $X'_{i}=m_{i}P({\Psi_{i}})$, $m_{i}>0$,  then 
$(\Gamma\setminus \Gamma')\cap \Psi_{i}=\set{\alpha_{i,0}}$ 
where $-\ck \alpha_{i,0}$ is  the   highest root of $\ck\Psi_{i}$ with 
respect to its simple system $\ck{(\Psi_{i}\cap \Gamma')}$ and $K_{i}=m_{i}$.
\end{conds}
Here, if $X'_{i}\neq 0$ and $\sum _{\alpha\in \Psi_{i}\cap  \Gamma}c_{\alpha}\alpha=0$ 
with all  $c_{\alpha}\in \real$ and  $c_{\alpha_{i,0}}=1$, we write
$K_{i}:=\sum_{\alpha\in \Psi_{i}\cap \Gamma}c_{\alpha}f(\alpha)$.
\end{lemma} 
 \begin{proof}   When $\real $ is replaced by $\Int$, this is implicit  in the proof 
 of Theorem \ref{thm:rootcoset};
 see especially \eqref{eq:compare}. The proof for $\real$ is essentially the same.
\end{proof}

 \subsection{}   Let $\wh \RR_{(\Psi,X')}(\Phi,\real)$ denote the set of all triples 
 $x=(\Gamma,\Gamma',f)$ such that $(\Gamma,f)\in \RR(\Phi,\real)$ 
 and the conditions (i)--(iv) above hold. 
  To $x=(\Gamma,\Gamma',f)\in \wh \RR_{(\Psi,X')}(\Phi,\real)$, we attach a point 
$v_{x}:= \sum_{\alpha\in \Gamma'}f(\alpha)\omega_{\alpha}(\Gamma')\in P_{\real}(\Psi)=\real \Psi$.
\begin{lemma}  \label{lem:bijlem2}
 Let $B$ be the 
(unique)  alcove of $W'$ in $\real\Psi$ such that $0\in \overline{B}$ and
$B$ is in the fundamental chamber for $W_{\Psi} $ with respect to its 
positive system $\Gamma'$. Then  
\begin{num}
\item We have $j_{\real}(\Gamma,f)=(\Psi,v_{x}+X')$.
\item  The coset $v_{x}+X'\in \real\Psi/X'$ depends only on $(\Gamma,f)$ and not on   $\Gamma'$.
\item   For fixed $(\Gamma,\Gamma')$, the  map  $x=(\Gamma,\Gamma',f)\mapsto v_{x}$  
defines a bijection from the set of triples $(\Gamma,\Gamma',f)
\in \wh \RR_{(\Psi,X')}(\Phi,\real) $ to the lower closure  $B_{0}$ of $B$. 
\end{num}
\end{lemma}
\begin{proof} For (a), observe that $\rh\Phi^{a}(\Gamma,f)=\rh\Phi^{a}(\Psi,v_{x}+X')$. 
Then (b) follows from (a).  For (c), note that the only restrictions on the 
function  $f$ are that
$f(\alpha)\geq 0$ for all $\alpha\in \Gamma$, that $f(\alpha)>0$ if 
$\alpha\in -\Phi_{+}\cap \Gamma$, and that 
$\sum_{\alpha\in \Gamma_{i}}c_{\alpha}f(\alpha)=m_{i}$ 
if $\Gamma_{i}'\neq \Gamma_{i}$. One checks using this that  
the set of points $v_{x}$ arising as in (b)  is   the set of all 
solutions $v$ of the following system of simultaneous inequalities 
for $v\in \real \Psi$: 
\bee  \begin{cases} \mpair{v,\alpha}\geq 0,&\text{for  $\alpha\in \Gamma'\cap \Psi_{+}$} \\
\mpair{v,\alpha}> 0,&\text{for  $\alpha\in \Gamma'\cap -\Psi_{+}$}\\
\mpair{v,\alpha}\geq -m_{i},&\text{for  $\alpha\in  (\Gamma\setminus \Gamma')\cap \Psi_{i,+}$} \\
\mpair{v,\alpha}> -m_{i},&\text{for  $\alpha\in  (\Gamma\setminus \Gamma')\cap -\Psi_{i,+}$} ,\\
\end{cases}\eee where $\Psi_{+}=\Psi\cap \Phi_{+}$, etc.
and $i$ in the last two inequalities is such that $\Gamma_{i}\neq \Gamma'_{i}$.
On the other hand, the alcove $B$ is defined by the  set of inequalities
above with each non-strict inequality replaced by strict inequality; the 
equations obtained by replacing inequalities by equalities define the 
walls of $B$. From the definition of lower closure,  the set of 
inequalities above also defines  $B_{0}$.
\end{proof}
\subsection{}\label{ss:defsqandg} There is a  map
 $q_{k}\colon \wh \RR_{(\Psi,X')}(\Phi,k)\rightarrow  \RR_{(\Psi,X')}(\Phi,k)$
 defined by $q(\Gamma,\Gamma',f)=(\Gamma,f)$. Write $q=q_{\Int}$. 
 Next define a function  $g\colon D'\rightarrow  \wh \RR_{(\Psi,X')}(\Phi,\real)$ 
 as follows. Let $d\in D'$. There is a unique alcove $B\in \mc{B}$ such that
 $0\in \overline{B}$ and $d\in B_{0}$. Let   $\Gamma'$ be the unique simple 
 system for $\Psi$ such that $B$ is contained in the fundamental chamber of
 $W_{\Psi}$ on $\real \Psi$ with respect to $\Gamma'$. There is a unique subset
 $\Gamma$ of $\Psi$ such that $(\Gamma,\Gamma')$ satisfies  the conditions
 of Lemma \ref{lem:bijlem1}(i)--(iv)  with the condition $K_{i}=m_{i}$ omitted 
 in (iii)--(iv).
 By Lemma \ref{lem:bijlem2}, there is a unique function 
 $f\colon \Gamma\rightarrow \real$ such that 
 $x:=(\Gamma,\Gamma',f)\in \RR_{(\Psi,X')}(\Phi,\real)$ and $v_{x}=d$.
 We define $g(d)=x$.
 
It is a simple consequence of  Lemmas \ref{lem:lowclos}, 
\ref{lem:bijlem1} and \ref{lem:bijlem2}  that  the map $q$ is surjective and
$g$ is a bijection.

\subsection{}We now come to the main result of this section.

\begin{theorem}\label{thm:bij}
 The  bijection $ \RR'_{(\Psi,X')}(\Phi,\real)\rightarrow  \RR_{(\Psi,X')}(\Phi,\real)$ given 
by restriction of $j_{\real}^{-1}$ is given by 
$j_{\real}^{-1}(\Psi,X)=q  g  k   h(X)$ for $X\in P_{\real}(\Psi)/X'$,
where $q$ and $g$ are the maps defined in \S\ref{ss:defsqandg}, $k$ is defined in 
\S\ref{ss:defk} and $h$ in \S\ref{ss:defh}.
\end{theorem}
\begin{proof} Write  $h(X)=d$, so $X=d+X'$. Set $g(d)=d'$, so 
$d-d'\in Y'\subseteq X'$ and $X=d'+X'$.
Set $g(d')=(\Gamma,\Gamma',f)=x$. By definition of $g$, we have 
$v_{x}=d'$ in Lemma \ref{lem:bijlem2}. Then $j_{\real}qgkh(X)=j_{\real}(\Gamma,f)=(\Psi, v_{x}+X')=(\Psi,d'+X')=(\Psi,X)$, 
and the Theorem follows.\end{proof}

\subsection{} If  $k=\Int$ and $X'$ has finite index in $P(\Psi)$, then evidently
\bee \vert  \RR'_{(\Psi,X')}(\Phi,\Z)\vert=
\vert   \RR_{(\Psi,X')}(\Phi,\Int)\vert \in \Nat.\eee 
This equality implies
various identities in $\Int$ by counting  the points of $ \RR'_{(\Psi,X')}(\Phi,\Z)$ and 
$   \RR_{(\Psi,X')}(\Phi,\Int)$    according to the definitions. 
We shall formulate these identities explicitly below  in the special 
case that  $\Psi=\Phi$ is irreducible; in general, the identities easily 
reduce to those in this special case, by considering indecomposable components of $\Psi$.

Assume that $\Psi=\Phi$ is indecomposable  and  $X'\neq \emptyset$. 
We have either $X'=M P(\Phi)$ or $X'=M P(\Phi^{\circ})$ for some
$M\in \Nat_{>0}$, so $[P(\Phi):X']$ is finite.
Fix the \nip $\Gamma:=\Pi\cup\set{-\omega}$ of $\Phi$  where $\omega$ 
is the highest root of $\Phi$ with respect to $\Pi$ if $X'=P(\Phi)$, and 
$-\ck \omega$ is the highest root of  $\ck \Psi$ with respect to
$\ck \Pi$ if $X'=P(\Psi^{\circ})$.

\begin{lemma} \label{lem:nipstab}
 Let $W$ act 
on \nips $\Sigma\subseteq \Phi$ by $(w,\Sigma)\mapsto w(\Sigma)$. 
Then the setwise stabiliser of $\Gamma$ in $W$  satisfies 
 $\text{\rm Stab}_{W}(\Gamma)\cong    P(\Phi)/Q(\Phi)$ and in particular 
 $\vert \text{\rm Stab}_{W}(\Gamma)\vert =f_{\Phi}$.\end{lemma}
\begin{proof} Suppose  that $\omega$ is the highest root of $\Phi$
(the other case follows from this one by considering the \nip  $\ck \Gamma\subseteq \ck \Phi$).
Let $a\colon \wt W^{a}\rightarrow W$ denote the natural projection.
Using that $\pi(\wt \Pi)=\Gamma$, it follows  that $a$ restricts to a map 
$a'\colon  \text{\rm Stab}_{\wt W^{a}}(\wt \Pi) \rightarrow 
\text{\rm Stab}_{W}(\Gamma)$. Using the definition of $\wt W^{a}$, 
one readily checks that $a'$ is surjective. The   group
 $ \text{\rm Stab}_{\wt W^{a}}(\wt \Pi)=  \text{\rm Stab}_{\wt W^{a}}(A)$ is discussed in detail
(to the extent that its elements are explicitly described)
 in \cite[Ch Vi, \S 2, no. 3]{Bour}. From the results there, one has  in particular that
$ \text{\rm Stab}_{\wt W^{a}}(\ck \Pi)\cong   P(\Phi)/Q(\Phi)$  and 
the restriction of $a$  
 to  $ \text{\rm Stab}_{\wt W^{a}}(\wt \Pi)$ is injective. Hence  
 \bee 
 \text{\rm Stab}_{W}(\Gamma)=
 a( \text{\rm Stab}_{\wt W^{a}}(\wt \Pi))\cong  
 \text{\rm Stab}_{\wt W^{a}}(\wt \Pi) \cong  P(\Phi)/Q(\Phi)
 \eee 
 and the result follows.
\end{proof}
\begin{remark} \label{rem:nipstab}
By considering irreducible components of $\Sigma$,
one can determine (using  the proof of the  above Lemma) the setwise  
stabiliser in $W_{\Sigma}$ of any \nip $\Sigma$ of $\Phi$. We do not go into 
the more delicate question of the stabiliser of $\Sigma $ in $W$.
\end{remark} 
\subsection{} Consider the integers $c_{\alpha}\in \Nat_{\geq 0}$ which
are defined by 
$\sum_{\alpha\in \Gamma}c_{\alpha}\alpha=0$ and $c_{\alpha}=1$.
Define a generalized ``ordered partition  function'' $\p \colon \Int\rightarrow  \Nat$ 
by 
\[\begin{split} 
\p(M)&:=\vert \mset{f\colon \Gamma\rightarrow 
\Nat\mid \sum_{\alpha\in \Gamma}c_{\alpha}f(\alpha)=M}\\ 
&=\vert \mset{f\colon \Gamma'\rightarrow 
\Nat\mid \sum_{\alpha\in \Gamma'}c_{\alpha}f(\alpha)\leq M}.
\end{split}
\]
Define some non-standard descent statistics on $W$  as follows. For $w\in W$, 
define $d(w)=\mset{\alpha\in \Gamma\mid w(\alpha)\in \Phi_{-}}$.
For $i\in \Nat$,  let 
$d_{i}:=\vert \mset{w\in W\mid \sum_{\alpha\in d(w)}c_{\alpha}=i}\vert$.
 Set $h=\sum_{\alpha\in \Gamma}c_{\alpha}$ ($h$ is the 
 Coxeter number of $\Phi$ if $\omega$
 is the highest root of $\Phi$, by \cite[Ch VI, \S 1, no 11, Prop. 31]{Bour}).
 Then $d_{i}=0$ if  $i=0$ or $i\geq h$. 
 \begin{proposition}   
 Let $N'=\vert \Pi_{\text{\rm long}}\vert$ and $N:=\vert \Pi\vert$. For any $M\in \Nat_{>0}$,
 \bee \sum_{i=0}^{h}\frac{d_{i}}{f_{\Phi}}\cdot \p(M-i)=
 \begin{cases}
M^{N },&\text{if $\omega\in \Phi_{\text{\rm long}}$}\\
 M^{N}k_{\Phi}^{N'},&\text{if $\omega\in \Phi_{\text{\rm short}}$.} 
 \end{cases}
 \eee 
 where   $\frac{d_{i}}{f_{\Phi}}\in \Nat$ for all $i$.
 \end{proposition}
 \begin{proof} For any $\Delta\subseteq \Gamma$, define
 $ d_{\Delta}:=\vert\mset{w\in W\mid d(w)=\Delta }\vert $.
 Note that $d_{\emptyset}=d_{\Gamma}=0$.
To prove the formula, it will suffice to show that 
  \bee 
  \frac{1}{f_{\Phi}}\sum_{\Delta\subseteq 
  \Gamma}d_{\Delta}\cdot \p\Bigl(M-\sum_{\alpha\in \Delta}c_{\alpha}\Bigr)=
  \begin{cases}
M^{N },&\text{if $\omega\in \Phi_{\text{\rm long}}$}\\
 M^{N}k_{\Phi}^{N'},&\text{if $\omega\in \Phi_{\text{\rm short}}$.} 
 \end{cases}
 \eee
 
 The right hand side is 
 $\vert D\cap P(\Phi)\vert= \vert  \RR'_{(\Psi,X')}(\Phi,\Z)\vert$.
 On the other hand, consider  the set  $  \wh \RR_{(\Psi,X')}(\Phi,\Z)$. 
 From Lemma \ref{lem:bijlem1},  one sees that
 a typical element of this set is of the form   $(w(\Gamma),w(\Gamma'), 
 f\circ w^{-1}\colon w(\Gamma)\rightarrow \Nat)$ for  (unique)
  $w\in W$ and $f\colon \Gamma \rightarrow \Nat$  such 
  that $f(\alpha)>0$ if $w(\alpha)\in \Phi_{-}$ and 
  $\sum_{\alpha\in \Gamma}f(\alpha)c_{\alpha}=M$.
 For fixed $w$, the number of  possible functions 
 $f$ is therefore $\p (M-\sum_{\alpha\in \Delta}c_{\alpha})$ where 
 $\Delta=\mset{\alpha\in \Gamma\mid w(\alpha)\in \Phi_{-}}$.  Hence 
 the left hand side is equal to $\frac{1}{f_{\Phi} }
 \vert  \wh \RR_{(\Psi,X')}(\Phi,\Z)\vert$.
 But since $W$ acts simply transitively on the simple systems
 $\Gamma'$ of $\Phi$, we get  from Lemma \ref{lem:nipstab} 
 that each  fiber of the map  $q_{\Int}\colon 
  \wh  \RR_{(\Psi,X')}(\Phi,\Z)\rightarrow   \RR_{(\Psi,X')}(\Phi,\Z) $ given by
  $(\Gamma,\Gamma',f)\mapsto (\Gamma,f)$  has cardinality $f_{\Phi}$  and so
   \bee \vert  \wh  \RR_{(\Psi,X')}(\Phi,\Z)\vert = 
   f_{\Phi} \vert   \RR_{(\Psi,X')}(\Phi,\Z)\vert =
   f_{\Phi} \vert   \RR'_{(\Psi,X')}(\Phi,\Z)\vert.
   \eee  
   This completes the proof of the formula.
   
   It remains to show that  $\frac{d_{i}}{f_{\Phi}}\in \Nat$ for all $i$. 
Let $G:=\text{\rm Stab}_{W}(\Gamma)$. One readily checks that
for $w\in W$, $g\in G$ and $\alpha\in \Gamma$, we have $d(wg)=g^{-1}d(w)$ 
and $c_{\alpha}=c_{g(\alpha)}$. It follows readily from this that
$d'_{i}:=\mset{w\in W\mid \sum_{\alpha\in d(w)}c_{\alpha}=i}$ is a
union of left cosets of $G$ (i.e. satisfies $d'(i)=d'(i)G$) and 
so its order $d_{i}$ is divisible by $\vert G\vert =f_{\Phi}$.  
   \end{proof}
   \begin{remark} Using the longest element of $W$, one easily shows that the sequence $(d_{0}, d_{1},\ldots ,d_{h})$ is 
   symmetric i.e $d_{i}=d_{h-i}$ for $i=0,\ldots, h$. One might conjecture that this sequence is
  also  strictly unimodal i.e. $d_{0}<d_{1}<\ldots <  d_{k}$ where $k$ is the largest integer satisfying
   $2k\leq h-1$.    \end{remark}

 \subsection{}  We discuss the identity in the previous Lemma in the 
 case $\Phi$ is of type $A_{n}$, $n\geq 1$. We identify $W$ with the 
 symmetric group $S_{n+1}$ on $\set{1,\ldots,n+1}$. 
 Consider the sequence $(i_{0,}i_{1},i_{2},\ldots, i_{n+1}):=(n+1,1,2,\ldots, n+1)$.
 For $\sigma\in W$, let 
 \[
 d_{\sigma}=\vert  \mset{k\in \Nat\mid1\leq  k\leq n+1, 
 \sigma(i_{k-1})>\sigma(i_{k})}\vert 
 \] 
 denote the number of cyclically 
 consecutive descents of $\sigma$.
 For $i=0,\ldots, n+1$, let $d_{i}:=
 \vert \mset{\sigma\in W\mid d_{\sigma}=i}\vert$, so $d_{0}=d_{n+1}=0$.
 The identity becomes
 \[  M^{n}=\sum_{i=1}^{n} \frac{d_{i}}{n+1}\binom{M+i-1}{n} \]
 for $M\in \Nat_{\geq 1}$.  For $n=2$, this  is the well known 
 identity expressing square numbers as  a sum of two consecutive triangular numbers. 

\bibliography{dlweylsubx}
\bibliographystyle{plain}

\nobreak

\end{document}